\documentclass{amsart}

\usepackage[colorlinks=true, urlcolor=black, citecolor=black, linkcolor=black, hyperfootnotes=true]{hyperref}
\usepackage{amssymb}
\usepackage{aliascnt}%for theorem counters
\usepackage{accents}
\usepackage{mathrsfs}

\numberwithin{equation}{section}

\newtheorem{thm}{Theorem}[section]

\newaliascnt{prp}{thm}
\newtheorem{prp}[prp]{Proposition}
\aliascntresetthe{prp}

\newaliascnt{cor}{thm}
\newtheorem{cor}[cor]{Corollary}
\aliascntresetthe{cor}

\newaliascnt{lem}{thm}
\newtheorem{lem}[lem]{Lemma}
\aliascntresetthe{lem}

\theoremstyle{definition}

\newaliascnt{dfn}{thm}
\newtheorem{dfn}[dfn]{Definition}
\aliascntresetthe{dfn}

\newaliascnt{xpl}{thm}
\newtheorem{xpl}[xpl]{Example}
\aliascntresetthe{xpl}

\newaliascnt{rmk}{thm}

\aliascntresetthe{rmk}

\newtheorem*{rmk*}{\textnormal{\emph{Remark}}}

\author{Tristan Bice}
\thanks{The author is supported by the GA\v{C}R project 22-07833K and RVO: 67985840.}
\address{Institute of Mathematics of the Czech Academy of Sciences, \v{Z}itn\'a 25, Prague}
\email{bice@math.cas.cz}

\makeatletter
\@namedef{subjclassname@2020}{\textup{2020} Mathematics Subject Classification}
\makeatother

\title{Sections of Fell Bundles over \'Etale Groupoids}

\begin{document}

\begin{abstract}
We construct the reduced and essential C*-algebra of a Fell bundle over an \'etale groupoid (in full generality, without any second countability, local compactness or Hausdorff assumptions, even on the unit space) directly from sections under convolution.  This eliminates the $j$-map commonly seen in the groupoid and Fell bundle C*-algebra literature.  We further show how multiplier algebras and other Banach *-bimodules can again be constructed directly from sections.  Finally, we extend Varela's morphisms from C*-bundles to Fell bundles, thus making the reduced C*-algebra construction functorial.
\end{abstract}

\maketitle

\section*{Introduction}

\subsection*{Background}

C*-algebras defined from \'etale groupoids have been playing an increasingly prominent role in C*-algebra theory since their inception in \cite{Renault1980}.  Twisted groupoids have likewise come to the fore, particularly since the advent of Kumjian-Renault's Weyl groupoid construction in \cite{Kumjian1986} and \cite{Renault2008}, one of the most promising noncommutative extensions of the classic Gelfand representation to be developed in the past century.  In recent years, attention has also been directed to more general Fell bundles $\rho:B\twoheadrightarrow\Gamma$ over \'etale groupoids $\Gamma$ and their resulting reduced C*-algebras $\mathcal{C}^*_\mathsf{r}(\Gamma)$ and noncommutative Cartan subalgebras $\mathcal{C}_0(\Gamma^0)$ (e.g. see \cite{Exel2011}, \cite{KwasniewskiMeyer2020}) with the potential of producing yet more powerful versions of Kumjian-Renault's Weyl groupoid construction.

Nevertheless, there is one basic issue with the construction of the reduced C*-algebra which has never been completely resolved, namely whether $\mathcal{C}^*_\mathsf{r}(\Gamma)$ really can be identified with a concrete algebra $\mathcal{C}^\rho_\mathsf{r}$ of sections under convolution.  To be sure, by definition, $\mathcal{C}^*_\mathsf{r}(\Gamma)$ contains a dense subalgebra identified with $\mathcal{C}^\rho_\mathsf{c}=\mathcal{C}_\mathsf{c}(\Gamma)$, an algebra consisting of compactly supported sections.  However, general elements of $\mathcal{C}^*_\mathsf{r}(\Gamma)$ are merely operators on some Hilbert module.  After the fact, one can at least identify $\mathcal{C}^*_\mathsf{r}(\Gamma)$ with a subspace of sections under a canonical `$j$-map'.  But even so, it is not always clear that the product of operators in $\mathcal{C}^*_\mathsf{r}(\Gamma)$ can be identified with the convolution of the corresponding sections (see the related work mentioned below for further discussion on this point).  As a result, precise arguments about $\mathcal{C}^*_\mathsf{r}(\Gamma)$ have inevitably involved an uncomfortable little dance, first restricting attention to $\mathcal{C}^\rho_\mathsf{c}$ under convolution and then using more technical arguments involving the $j$-map to extend to the entire reduced C*-algebra $\mathcal{C}^*_\mathsf{r}(\Gamma)$.

The primary goal of the present paper is to show this state of affairs is entirely avoidable by constructing the reduced C*-algebra directly from concrete sections.  Moreover, our construction is very elementary, using nothing more sophisticated than the Cauchy-Schwarz inequality for Hilbert modules.  Specifically, in our approach, we first define a `$\mathsf{b}$-norm' on arbitrary sections of the Fell bundle.  This uses the convolution product, but only when one of the sections is finitely supported, which is always well-defined.  Then we show that the sections with finite $\mathsf{b}$-norm form a Banach space including $\mathcal{C}^\rho_\mathsf{c}$.  The reduced C*-algebra $\mathcal{C}^\rho_\mathsf{r}$ is then defined to be the closure of $\mathcal{C}^\rho_\mathsf{c}$ within this Banach space -- it is then just a matter of showing convolution is still well-defined on the closure.  In fact, we will show that convolution is well-defined on even larger algebras and modules of sections.  For example, this allows even the multiplier algebra of $\mathcal{C}^\rho_\mathsf{r}$ to be constructed directly from concrete sections, at least when the total space of the bundle has a full collection of units.

Another advantage of our approach is that we can avoid unnecessary topological restrictions on the groupoids that are often seen elsewhere.  Indeed, traditionally the focus has been on second countable locally compact Hausdorff \'etale groupoids.  However, recent work shows that second countability is unnecessary and only the unit space needs to be Hausdorff for much of the same theory to apply.  But as soon as we start dealing with bundles where the norm is only upper semicontinuous, there are non-trivial examples where even the unit space fails to be Hausdorff.

For example, given any $\mathsf{T}_1$ space $X$, we can form a C*-bundle $\rho:X\times\mathbb{C}\twoheadrightarrow X$ from the usual projection but the (unusual) topology generated by the basis
\[O^{a,\varepsilon}_x=(\{x\}\times B^{a,\varepsilon})\cup((O\setminus\{x\})\times B^{0,\varepsilon}),\]
where $O\subseteq X$ is open, $a\in\mathbb{C}$, $\varepsilon>0$, $x\in X$ and $B^{a,\varepsilon}=\{b\in\mathbb{C}:|a-b|<\varepsilon\}$.  Then we have continuous sections $\delta_x$ with $\delta_x(x)=(x,1)$ and $\delta_x(y)=(y,0)$ for $y\neq x$.  In particular, the norm functions of continuous sections always distinguish points of $X$, even when $X$ is highly non-Hausdorff, e.g. the cofinite topology on an infinite set.  This contrasts with continuous sections of trivial bundles or even bundles where the norm is continuous, as continuous sections must then have the same norm on pairs of points that can not be separated by disjoint neighbourhoods.

As we shall see, C*-algebras can still be constructed from Fell bundles like this in much the same way.  We just have to be careful to replace the compact bisections considered in the Hausdorff case with `bicompact' $K$ in the general case -- bicompactness just means that $K$ is a biclosed compact subset of an open bisection (biclosed meaning that $\mathsf{s}[K]$ and $\mathsf{r}[K]$ are both closed in the unit space).

\subsection*{Related Work}

The present paper is a thoroughly revised and extended version of the first part of \cite{Bice2021DHKR}.  Duwenig, Williams and Zimmerman achieved much the same thing in \cite{DuwenigWilliamsZimmerman2022}, at least for the reduced C*-algebra of saturated Fell bundles over second countable locally compact Hausdorff \'etale groupoids.  Their approach is more traditional, still relying on the left regular representation to construct $\mathcal{C}^*_\mathsf{r}(\Gamma)$, but then showing that the canonical $j$-map really does respect convolution.  This can be viewed as a generalisation of \cite[Proposition 2.8]{BrownFullerPittsReznikoff2021} for twisted groupoids (even ones that are not second countable), as these correspond to continuous Fell line bundles.  The special case of cocycle twists over second countable groupoids was even dealt with in Renault's foundational work in \cite[Proposition II.4.2]{Renault1980}.  Our approach seems more suitable for multiplier algebras although it would certainly be interesting to know if there might be an analogous  $j$-map approach for these too.

Recently in \cite{Taylor2023}, Taylor has also presented his own approach to the theory developed in \cite{Bice2021DHKR}.  The last section of the present paper on morphisms is thus analogous to \S2 and \S4 of \cite{Taylor2023}.  While \cite{Bice2021DHKR} only dealt with Hausdorff groupoids, in \cite{Taylor2023} only the unit space is required to be Hausdorff.  The present paper is more general still, making no Hausdorff assumption even on the unit space.

%\newpage

\subsection*{Motivation}

To motivate our approach, let us first consider a much simplified situation where we just want to define a C*-algebra from a topological space $X$, which for simplicity we may also assume is locally compact and Hausdorff.  The usual way of doing this would be to just consider $\mathcal{C}_0(X)$, the continuous complex-valued functions on $X$ vanishing at infinity, which is immediately seen to form a commutative C*-algebra under pointwise operations and the supremum norm.

Another way of doing this would be as follows.  Start with the smaller normed *-algebra $\mathcal{C}_\mathsf{c}(X)$ consisting of continuous complex-valued functions with compact supports.  Next note that, for each $x\in X$, we have a representation of this algebra on $\mathbb{C}$, considered as a one-dimensional Hilbert space, given by $f\cdot\lambda=f(x)\lambda$, for all $f\in\mathcal{C}_\mathsf{c}(X)$ and $\lambda\in\mathbb{C}$.  Combining these representations yields a faithful representation of $\mathcal{C}_\mathsf{c}(X)$ as a *-algebra of bounded operators on $\ell^2(X)$.  The closure of these operators in $\mathsf{B}(\ell^2)$ then forms a C*-algebra which we call the reduced C*-algebra of $X$, denoted by $\mathcal{C}^*_\mathsf{r}(X)$.  We can then define a map $j$ taking operators $a\in\mathcal{C}^*_\mathsf{r}(X)$ back to concrete functions on $X$, specifically where $a\delta_x=j(a)(x)\delta_x$, for all $x\in X$.  Finally, we can further verify that $j$ turns the usual operator sums, products, etc. into pointwise operations on functions so in the end we get a C*-algebra $j[\mathcal{C}^*_\mathsf{r}(X)]$ containing the original $\mathcal{C}_\mathsf{c}(X)$ as a dense subalgebra.

Anyone familiar with C*-algebras will quickly recognise this as nothing more than a ridiculously complicated way of defining $\mathcal{C}_0(X)$!  And yet this process is exactly what the standard reduced groupoid C*-algebra construction amounts to, at least in the special case of topological spaces.  Could there not be a better way?

Certainly, for an \'etale groupoid $G$, we can not simply define the reduced C*-algebra as $\mathcal{C}_0(G)$.   The convolution product may not even be well-defined on all pairs in $\mathcal{C}_0(G)$.  Indeed, the range of $\mathcal{C}^*_\mathsf{r}(G)$ under the $j$-map can be a proper subset of $\mathcal{C}_0(G)$ which may not admit any simple topological characterisation.  With this in mind, it definitely makes sense to build up the reduced C*-algebra of $G$ starting from $\mathcal{C}_\mathsf{c}(G)$ instead.  However, this can be done in an elementary way without the need to pass back and forth between functions and operators.

To see how this would work again in the special case of a topological space $X$ as above, we proceed as follows.  To begin with, consider the big C*-algebra $\ell^\infty(X)$ consisting of all bounded complex-valued functions, even the discontinuous ones.  Next note that compactly supported continuous functions are bounded so $\mathcal{C}_\mathsf{c}(X)\subseteq\ell^\infty(X)$.  Now just define $\mathcal{C}_0(X)$ to be the closure of $\mathcal{C}_\mathsf{c}(X)$ in $\ell^\infty(X)$.

As we are working in the big C*-algebra $\ell^\infty(X)$ here, we are also free to consider larger C*-subalgebras which may have some important relationship to $\mathcal{C}_0(X)$.  For example, we can consider all multipliers of $\mathcal{C}_0(X)$ within $\ell^\infty(X)$, i.e. those functions $f\in\ell^\infty(X)$ such that $fg\in\mathcal{C}_0(X)$, for all $g\in\mathcal{C}_0(X)$.  These turn out to be precisely the continuous bounded functions $\mathcal{C}_\mathsf{b}(X)$ which can indeed be identified with the usual multiplier algebra of $\mathcal{C}_0(X)$.

To do this with more general \'etale groupoids $G$ (and even more general Fell bundles), it is just a matter of first identifying the appropriate analog of the big C*-algebra $\ell^\infty(X)$.  For this we will have to consider several different norms and their relationship to the convolution product.  This takes a not insignificant amount of work, as we will see in the following sections, but we believe the pay-off is worth it.  For example, this again allows us to identify the multiplier algebra with a C*-algebra of functions on $G$ under convolution.

\section{Preliminaries}

First we review the necessary background on groupoids and their Fell bundles.

\subsection{Semigroupoids}\label{Semigroupoids}

Most of the material here is fairly standard, the only difference being that with Fell bundles it is convenient to work with object-free notions of semigroupoids, categories, groupoids, etc..

\begin{dfn}\label{Semigroupoids}
A \emph{semigroupoid} is a set $S$ together with a partial associative binary operation $(a,b)\mapsto ab$, i.e. $(ab)c$ is defined iff $a(bc)$ is, in which case they are equal.

A \emph{semigroup} is a semigroupoid $S$ where the product $ab$ is defined for all $a,b\in S$.
\end{dfn}

Semigroupoids are sometimes called `partial semigroups' (e.g. see \cite{BergelsonBlassHindman1994}).

The pairs for which the product is defined are denoted by
\[S^2=\{(a,b)\in S\times S:ab\text{ is defined}\}.\]
We call $x\in S$ a \emph{unit} if $xa=a$ when $(a,x)\in S^2$ and $ax=a$ when $(x,a)\in S^2$.  Let
\[S^0=\{x\in S:x\text{ is a unit}\}.\]

\begin{prp}
For any $a\in S$, there is at most one $x\in S^0$ with $(a,x)\in S^2$.
\end{prp}

\begin{proof}
Say we have $x,y\in S^0$ with $a=ax=ay$.  Then $a=ay=(ax)y=a(xy)$, by associativity.  In particular, $(x,y)\in S^2$ so $x=xy=y$, as $x,y\in S^0$.
\end{proof}

Likewise, there can be at most one $x\in S^0$ with $(x,a)\in S^2$.  When they exist, we call these the \emph{source unit} $\mathsf{s}(a)$ and \emph{range unit} $\mathsf{r}(a)$ respectively, i.e.
\[a\mathsf{s}(a)=a=\mathsf{r}(a)a.\]
For any $(a,b)\in S^0$, if $b$ has a source unit $\mathsf{s}(b)$ then $ab=a(b\mathsf{s}(b))=(ab)\mathsf{s}(b)$, showing that $\mathsf{s}(ab)=\mathsf{s}(b)$.  Likewise, if $a$ has a range unit then $\mathsf{r}(ab)=\mathsf{r}(a)$.

\begin{prp}
If $\mathsf{s}(a)$ and $\mathsf{r}(b)$ exist then
\[(a,b)\in S^2\qquad\Rightarrow\qquad\mathsf{s}(a)=\mathsf{r}(b).\]
\end{prp}

\begin{proof}
Note that $(a,b)\in S^2$ implies $ab=(a\mathsf{s}(a))b=a(\mathsf{s}(a)b)$, i.e. $(\mathsf{s}(a),b)\in S^2$, which then in turn implies $\mathsf{s}(a)b=\mathsf{s}(a)(\mathsf{r}(b)b)=(\mathsf{s}(a)\mathsf{r}(b))b$, i.e. $(\mathsf{s}(a),\mathsf{r}(b))\in S^2$, and hence $\mathsf{s}(a)=\mathsf{s}(a)\mathsf{r}(b)=\mathsf{r}(b)$.
\end{proof}

Categories are semigroupoids where the converse also holds.

\begin{dfn}
A \emph{category} is a semigroupoid where $\mathsf{s}(a)$ and $\mathsf{r}(a)$ always exist and
\begin{equation}\label{CategoryDef}
\mathsf{s}(a)=\mathsf{r}(b)\qquad\Rightarrow\qquad(a,b)\in S^2.
\end{equation}
\end{dfn}

In Fell bundles and C*-algebras, involutions also play a key role.

\begin{dfn}\label{*Semigroupoid}
A \emph{*-semigroupoid} is a semigroupoid $S$ with an involution, i.e. a map $a\mapsto a^*$ on $S$ such that $a^{**}=a$ and $(ab)^*=b^*a^*$, for all $(a,b)\in S^2$.

Likewise, a \emph{*-category}/\emph{*-semigroup} is a category/semigroup with an involution.
\end{dfn}

If $S$ is a semigroupoid, an \emph{inverse} of $a\in S$ is an element $a^{-1}\in S$ with $aa^{-1},a^{-1}a\in S^0$ (in particular, $(a,a^{-1}),(a^{-1},a)\in S^2$).  The \emph{core} of $S$ is
\[S^\times=\{a\in S:a\text{ has an inverse}\}.\]

\begin{dfn}
A \emph{groupoid} $G$ is a semigroupoid of invertibles, i.e. $G=G^\times$.
\end{dfn}

\begin{prp}
Every groupoid is a *-category when we take $g^*=g^{-1}$.
\end{prp}

\begin{proof}
Say $G$ is a groupoid, so every $g\in G$ has an inverse $g^{-1}$.  In particular, every $x\in G^0$ has an inverse $x^{-1}=xx^{-1}=x(xx^{-1})=x$.  Taking $x=g^{-1}g$, for some $g\in G$, it follows that $g^{-1}g=(g^{-1}g)(g^{-1}g)=g^{-1}(g(g^{-1}g))$ and hence $\mathsf{s}(g)=g^{-1}g$.  Likewise $\mathsf{r}(g)=gg^{-1}$ so, in particular, all elements have source and range units.  Moreover, if $\mathsf{s}(g)=\mathsf{r}(h)$ then $g=g\mathsf{s}(g)=g\mathsf{r}(h)=g(hh^{-1})=(gh)h^{-1}$ so $(g,h)\in G^2$, showing that $G$ is a category.  Now it follows that $g\mapsto g^{-1}$ is an involution which means $G$ is also a *-category when we take $g^*=g^{-1}$.
\end{proof}

A map $\rho:S\rightarrow T$ between semigroupoids $S$ and $T$ is a \emph{homomorphism} if
\[\tag{Homomorphism}(a,b)\in S^2\qquad\Rightarrow\qquad\rho(ab)=\rho(a)\rho(b).\]
In particular, $(a,b)\in S^2$ implies $(\rho(a),\rho(b))\in T^2$.  An \emph{isocofibration} is a homomorphism where the converse also holds, i.e. $(a,b)\in S^2$ whenever $(\rho(a),\rho(b))\in T^2$.  If $S$ and $T$ are *-semigroupoids and $\rho$ also respects the involution, i.e. $\rho(a^*)=\rho(a)^*$, then $\rho$ is a \emph{*-isocofibration}.  A \emph{functor} is a unital homomorphism, i.e. $\rho[S^0]\subseteq T^0$.

\begin{prp}\label{HomoFun}
Every homomorphism to a groupoid is automatically a functor.
\end{prp}

\begin{proof}
Say $S$ is a semigroupoid, $G$ is a groupoid and $\rho:S\rightarrow G$ is a homomorphism.  For any $a\in S^0$ or even any idempotent $a=aa$, we see that $\rho(a)=\rho(a)\rho(a)$ and hence $\rho(a)\rho(a)^{-1}=\rho(a)\rho(a)\rho(a)^{-1}=\rho(a)$, showing that $\rho(a)\in\Gamma^0$.
\end{proof}

\begin{cor}
Every homomorphism between groupoids is a *-functor.
\end{cor}

\begin{proof}
Say $G$ and $H$ are groupoids and $\rho:G\rightarrow H$ is a homomorphism.  For any $g\in G$, we have $\rho(g^{-1})\rho(g)=\rho(g^{-1}g)\in H^0$, by \autoref{HomoFun}, and hence
\[\rho(g^{-1})=\rho(g^{-1})\rho(g)\rho(g)^{-1}=\rho(g)^{-1}.\qedhere\]
\end{proof}

To define Fell bundles, we will also need the following topological notions.

\begin{dfn}\label{TopologicalGroupoids}
A \emph{topological category} is a category carrying a topology making the source $\mathsf{s}$, range $\mathsf{r}$ and the product (on $S^2$ as a subspace of $S\times S$) continuous.

A \emph{topological *-semigroupoid} is a *-semigroupoid $S$ carrying a topology making both the involution and product continuous.

A \emph{topological groupoid} is a groupoid that is topological as a *-semigroupoid with $a^*=a^{-1}$, i.e. where both the inverse and product are continuous.

An \emph{\'etale groupoid} is a topological groupoid where $\mathsf{s}$ (or $\mathsf{r}$) is an open map.
\end{dfn}

Recall that a \emph{slice} of a groupoid $\Gamma$ is a subset $S\subseteq\Gamma$ on which $\mathsf{s}$ and $\mathsf{r}$ are injective or, equivalently, such that $SS^{-1}\cup S^{-1}S\subseteq\Gamma^0$.  \'Etale groupoids are precisely those with a basis of open slices which is closed under pointwise products and inverses.

The following results will be needed later to define the singular ideal of the reduced C*-algebra (which is quotiented out to define the essential C*-algebra).  First recall that a subset $N$ of a topological space $X$ is \emph{nowhere dense} if $\overline{N}^\circ=\emptyset$.

\begin{prp}
If $\Gamma$ is an \'etale groupoid, $K$ is compact and $N\subseteq K\subseteq\Gamma$ then 
\begin{equation}\label{sNvsN}
\mathsf{s}[N]\text{ is nowhere dense}\qquad\Leftrightarrow\qquad N\text{ is nowhere dense}.
\end{equation}
\end{prp}

\begin{proof}
If $N$ is not nowhere dense in $\Gamma$ then $\mathsf{s}[\overline{N}^\circ]\neq\emptyset$.  As $\Gamma$ is \'etale, $\mathsf{s}$ is an open map and hence $\mathsf{s}[\overline{N}^\circ]$ is an open subset.  Also $\mathsf{s}$ is continuous so $\mathsf{s}[\overline{N}^\circ]\subseteq\mathsf{s}[\overline{N}]\subseteq\overline{\mathsf{s}[N]}$ and hence $\overline{\mathsf{s}[N]}^\circ\neq\emptyset$, showing that $\mathsf{s}[N]$ also fails to be nowhere dense.

Conversely, say $N$ is nowhere dense.  As $\Gamma$ is \'etale, we have a finite family of open slices $O_1,\ldots,O_n\subseteq\Gamma$ covering $N$.  As $\mathsf{s}$ is a homeomorphism on each $O_k$, it follows that $\mathsf{s}[O_k\cap N]$ is also nowhere dense in $\mathsf{s}[O_k]$ and hence in $\Gamma$, as being nowhere dense is the same as being nowhere dense in any open subspace.  Thus $\mathsf{s}[N]=\bigcup_k\mathsf{s}[O_k\cap N]$ is also nowhere dense.
\end{proof}

Actually, we will really be interested in the following corollary for \emph{meagre} subsets, i.e. countable (increasing) unions of nowhere dense subsets.  Recall that a space $X$ is \emph{$\sigma$-compact} if it is a countable union of compact subsets.

\begin{cor}
If $\Gamma$ is an \'etale groupoid, $K$ is $\sigma$-compact and $M\subseteq K\subseteq\Gamma$ then
\begin{equation}\label{sMeagre}
M\text{ is meagre}\qquad\Leftrightarrow\qquad\mathsf{s}[M]\text{ is meagre}.
\end{equation}
\end{cor}

\begin{proof}
If $\mathsf{s}[M]$ is meagre then it is a countable union of nowhere dense subsets, each of which has nowhere dense $\mathsf{s}$-preimage, as in the proof of \eqref{sNvsN} above.  As these preimages cover $M$, it follows that $M$ is meagre.

Conversely, if $M$ is meagre then $M$ is a countable union of nowhere dense subsets contained in compact subsets of $\Gamma$.  By \eqref{sNvsN}, the $\mathsf{s}$-image of each of these is nowhere dense and hence $\mathsf{s}[M]$ is meagre.
\end{proof}

Of course, the above results are also valid with $\mathsf{r}$ in place of $\mathsf{s}$.

In a general \'etale groupoid, arbitrary compact subsets need not be closed under taking products.  However, this will be true for particular kinds of compact subsets.

\begin{dfn}
Let $\Gamma$ be an \'etale groupoid.  We call $B\subseteq\Gamma$
\begin{enumerate}
\item \emph{biclosed} if both $\mathsf{s}[B]$ and $\mathsf{r}[B]$ are closed in $\Gamma^0$.
\item \emph{bicompact} if $B$ is a biclosed compact subset of an open slice.
\item \emph{$\sigma$-bicompact} if $B$ is a countable union of bicompact subsets.
\end{enumerate}
\end{dfn}

Note that if $\Gamma^0$ is Hausdorff then every compact subset is automatically biclosed and every compact slice is automatically bicompact, by \cite[Proposition 6.3]{BiceStarling2018} -- we just need bicompactness so our results remain valid for non-Hausdorff unit spaces.

\begin{prp}\label{CompactProducts}
If $\Gamma$ is an \'etale groupoid and $K,L\subseteq\Gamma$ are bicompact, so is
\[KL=\{kl:(k,l)\in(K\times L)\cap\Gamma^2\}.\]
\end{prp}

\begin{proof}
Take open slices $O,N\subseteq\Gamma$ with $K\subseteq O$ and $L\subseteq N$.  Then $\mathsf{s}[K]$ and $\mathsf{r}[L]$ are closed in $\Gamma^0$ and compact, as $\mathsf{s}$ and $\mathsf{r}$ are continuous.  It follows that $\mathsf{s}[K]\cap\mathsf{r}[L]$ is also compact and closed in $\Gamma^0$.  Thus $\mathsf{r}[KL]=\mathsf{r}[\mathsf{s}|_O^{-1}[\mathsf{s}[K]\cap\mathsf{r}[L]]]$ is also compact and closed in $\mathsf{r}[K]$ and hence in $\Gamma^0$, as $\mathsf{s}$ and $\mathsf{r}$ are homeomorphisms on $O$.  Likewise, $\mathsf{s}[KL]$ is closed in $\Gamma^0$ and $KL=\mathsf{r}|^{-1}_{ON}[\mathsf{r}[KL]]$ is compact, as $\mathsf{r}$ is a homeomorphism on $ON$, showing that $KL$ is a biclosed compact subset of the open slice $ON$.
\end{proof}

This leads to the following result, which we will use (only once, in the proof of \autoref{ComeagreEssentialNorm} below) to show that the essential seminorm of a section in the reduced C*-algebra can be calculated from the $\mathsf{b}$-norm of $a$ restricted to a subgroupoid on which $a$ is continuous.  Basically it says that, given any any meagre $M\subseteq\Gamma$ and any $S\subseteq\Gamma$ contained in a $\sigma$-bicompact subset of $\Gamma$, we can cover all except possibly a meagre subset of $S$ with a subgroupoid $G$ which avoids $M$.

\begin{lem}\label{ComeagreSubgroupoids}
If $\Gamma$ is an \'etale groupoid, $K$ is $\sigma$-bicompact, $S\subseteq K\subseteq\Gamma$ and $M\subseteq\Gamma$ is meagre then we have $\sigma$-bicompact $L$ and subgroupoids $G,H\subseteq L\subseteq\Gamma$ such that $S\subseteq G\cup H$, $G\cap M=\emptyset$, $H$ is meagre and $(\mathsf{s}[G]\cup\mathsf{r}[G])\cap(\mathsf{s}[H]\cup\mathsf{r}[H])=\emptyset$.
\end{lem}

\begin{proof}
By \autoref{CompactProducts}, the subgroupoid generated by any $\sigma$-bicompact subset of $\Gamma$ is again $\sigma$-bicompact.  In particular, if $\Delta$ is the subgroupoid of $G$ generated by $S$ then it is contained in the $\sigma$-bicompact subgroupoid $L$ generated by $K$.  Let $M_0=M\cap\Delta$ and recursively define $(X_n)$ and $(M_n)$ by
\[X_n=\mathsf{r}[M_n]\cup\mathsf{s}[M_n]\qquad\text{and}\qquad M_{n+1}=\Delta\cap(\mathsf{s}^{-1}[X_n]\cup\mathsf{r}^{-1}[X_n])\]
By \eqref{sMeagre}, $X_n$ and $M_n$ are meagre, for all $n\in\mathbb{N}$.  Thus $H=\bigcup_nM_n$ is also meagre and, for all $\delta\in\Delta$,
\[\delta\in H\qquad\Leftrightarrow\qquad\mathsf{s}(\delta)\in H\qquad\Leftrightarrow\qquad\mathsf{r}(\delta)\in H.\]
It follows that both $H$ and $G=\Delta\setminus H$ are subgroupoids of $\Delta$ and hence $\Gamma$ such that $(\mathsf{s}[G]\cup\mathsf{r}[G])\cap(\mathsf{s}[H]\cup\mathsf{r}[H])=\emptyset$.  Also certainly $S\subseteq\Delta=G\cup H$, as well as $G\cap M=M\cap\Delta\setminus H=M_0\setminus H=\emptyset$.
\end{proof}

\subsection{Banach Bundles}\label{BanachBundles}

In its most general sense, a \emph{bundle} is simply an open continuous surjection $\rho:B\twoheadrightarrow X$ between topological spaces $B$ and $X$.  Here $X$ is called the \emph{base} while $B$ is called the \emph{total space} of the bundle.  The \emph{fibres} of the bundle are the preimages of single points $B_x=\rho^{-1}\{x\}$.

We will be interested in the following special kinds of bundles.

\begin{dfn}
A \emph{Banach bundle} is a bundle $\rho:B\twoheadrightarrow X$ where
\begin{enumerate}
\item Each fibre is a (complex) Banach space.
\item For each $\lambda\in\mathbb{C}$, the map $x\mapsto\lambda x$ is continuous on $B$.
\item Addition is continuous on $B\times_\rho B=\{(a,b):\rho(a)=\rho(b)\}$.
\item The norm is upper semicontinuous on $B$.
\item\label{TopologicalCompatibility} $\rho(b_\lambda)\rightarrow x$ and $\|b_\lambda\|\rightarrow0$ implies $b_\lambda\rightarrow0_x$.
\end{enumerate}
\end{dfn}

Our primary reference here is \cite{DoranFell1988}, even though the Banach bundles in \cite{DoranFell1988} are assumed to have a Hausdorff base and continuous norm.  However, most of the basic results pass to more general Banach bundles without any significant change.  For example, one can check that the proofs of \cite[Proposition 13.10 and 13.11]{DoranFell1988} do not use the Hausdorff property and are still valid even when the norm is only upper semicontinuous.  These results thus imply that $(\lambda,b)\rightarrow\lambda b$ is automatically (jointly) continuous on $\mathbb{C}\times B$ and that the subspace topology on each fibre coincides with the norm topology.  We will be careful to indicate whenever a result really does require the norm to be continuous or the base space to be Hausdorff.

We will be particularly concerned with sections of Banach bundles $\rho:B\twoheadrightarrow X$.  First let $B^X$ denote the set of functions from $X$ to $B$, i.e. the subsets of $B\times X$ containing exactly one element of $B\times\{x\}$, for each $x\in X$.  Further let
\[\rho^{-1}=\{(b,x):(x,b)\in\rho\}=\{(b,x):\rho(b)=x\}.\]
The arbitrary sections of $\rho$ are then given by
\[\mathcal{S}^\rho=\{a\in B^X:a\subseteq\rho^{-1}\}=\{a\in B^X:\rho\circ a=\mathrm{id}_X\}.\]
Note that these sections $\mathcal{S}^\rho$ form a vector space with respect to pointwise operations.

Let us denote the continuity points of any function $a$ on $X$ by
\[\mathcal{C}(a)=\{x\in X:a\text{ is continuous at }x\}.\]
The continuity of sums and scalar products means that, for any $a,b\in\mathcal{S}^\rho$ and $z\in\mathbb{C}$,
\[\mathcal{C}(a)\cap\mathcal{C}(b)\subseteq\mathcal{C}(a+b)\qquad\text{and}\qquad\mathcal{C}(a)\subseteq\mathcal{C}(za).\]
Consequently, for any $Y\subseteq X$, the sections that are supported in $Y$ and continuous at every point of $Y$ form a subspace of $\mathcal{S}^\rho$, which we denote by
\[\mathcal{C}^\rho(Y)=\{a\in\mathcal{S}^\rho:\mathrm{supp}(a)\subseteq Y\subseteq\mathcal{C}(a)\},\]
where $\mathrm{supp}(a)=\{x\in X:a(x)\neq0_x\}$.  Note $Y\subseteq\mathcal{C}(a)$ implies that the restriction $a|_Y$ is a continuous function on $Y$, but the converse only holds when $Y$ is open.  Also, we omit $Y$ whenever $Y=X$, i.e.
\[\mathcal{C}^\rho=\mathcal{C}^\rho(X)=\{a\in\mathcal{S}^\rho:a\text{ is continuous everywhere}\}.\]

Even though the norm of a Banach bundle is only required to be upper semicontinuous, we can always ensure that the norm on the range of some continuous section is continuous on some comeagre(=complement of meagre) subset.

For any $a\in\mathcal{S}^\rho$, define its norm function $a_\infty:X\rightarrow[0,\infty)$ by $a_\infty(x)=\|a(x)\|$.

\begin{prp}\label{ComeagreNormContinuity}
For any Banach bundle $\rho:B\twoheadrightarrow X$ and continuous section $a\in\mathcal{C}^\rho$, the continuity points $\mathcal{C}(a_\infty)$ of the norm function $a_\infty$ are comeagre.
\end{prp}

\begin{proof}
As $a$ is continuous and the norm is upper semicontinuous, $a_\infty$ is upper semicontinuous.  Define the oscillation of any $O\subseteq X$ by
\[\mathrm{osc}(O)=\sup_{x\in O}\|a(x)\|-\inf_{x\in O}\|a(x)\|\]
(in particular, $\mathrm{osc}(O)=-\infty$ if $O=\emptyset$, while $\mathrm{osc}(O)=\infty$ if $a$ is unbounded on $O$).  For any $n\in\mathbb{N}$, let
\[O_n=\bigcup\{O:O\subseteq X\text{ is open and }\mathrm{osc}(O)\leq\tfrac{1}{n}\}.\]
We claim that $O_n$ is dense in $X$.  If not, take $x\in X\setminus\mathrm{cl}(O_n)$.  As $a_\infty$ is upper semicontinuous, we have an open neighbourhood $O\subseteq X\setminus\mathrm{cl}(O_n)$ of $x$ such that $\|a(y)\|<\|a(x)\|+\frac{1}{2n}$, for all $y\in O$.  As $\mathrm{osc}(O)>1/n$ (because $O\nsubseteq O_n$) we must then have $y\in O\subseteq X\setminus\mathrm{cl}(O_n)$ with $\|a(y)\|<\|a(x)\|-\frac{1}{2n}$.  We can then repeat this procedure with $y$ etc. to show that $\|a(z)\|<0$, for some $z\in X\setminus\mathrm{cl}(O_n)$, contradicting the fact that $\mathrm{ran}(a_\infty)\subseteq[0,\infty)$.  This proves the claim and hence $Y=\bigcap_nO_n$ is a comeagre subset of points at which $a_\infty$ is continuous.
\end{proof}

Let us consider the uniform norm on $\mathcal{S}^\rho$ with values in $[0,\infty]$ defined by
\[\|a\|_\infty=\sup_{x\in X}a_\infty(x)=\sup_{x\in X}\|a(x)\|.\]
As usual, this defines a metric $\|a-b\|_\infty$ again with values in $[0,\infty]$.  As each fibre is complete, $\mathcal{S}^\rho$ is uniformly complete and hence the subspace of sections
\[\mathcal{S}^\rho_\infty=\{a\in\mathcal{S}^\rho:\|a\|_\infty<\infty\}\]
with finite uniform norm forms a Banach space.

For any $\delta\geq0$, we define the \emph{$\delta$-support} of $a\in\mathcal{S}^\rho$ by
\[\mathrm{supp}_\delta(a)=\{x\in X:\|a(x)\|\geq\delta\}.\]
We denote the sections in $\mathcal{C}^\rho(Y)$ \emph{vanishing at infinity} by
\[\mathcal{C}_0^\rho(Y)=\{a\in\mathcal{C}^\rho(Y):\forall\delta>0\ (\mathrm{supp}_\delta(a)\text{ is compact})\}.\]
Let us show these form uniformly closed subspaces of continuous sections which can often be approximated by increasing sequences of compactly supported ones.

\begin{prp}\label{C0}
Take any Banach bundle $\rho:B\twoheadrightarrow X$, $Y\subseteq X$ and $a\in\mathcal{C}_0^\rho(Y)$.
\begin{enumerate}
\item\label{C01} $\mathcal{C}_0^\rho(Y)$ is always a uniformly closed subspace of $\mathcal{S}^\rho_\infty$.
\item\label{C02} If $Y$ is also locally compact and Hausdorff then we have $(a_n)\subseteq\mathcal{C}^\rho(Y)$ converging uniformly to $a$ and subcompacta $(K_n)$ of $\Gamma$ such that
\[Y\supseteq K_n\supseteq\mathrm{supp}(a_n)\subseteq\mathrm{supp}(a_{n+1})\subseteq\mathrm{supp}(a),\text{ for all }n\in\mathbb{N}.\]
\item If the norm is also continuous on $B$ then above we even ensure that
\[\mathrm{supp}(a_n)\subseteq K_n\subseteq\mathrm{supp}(a_{n+1}).\]
\end{enumerate}
\end{prp}

\begin{proof}\
\begin{enumerate}
\item By the definition of $\mathcal{C}_0^\rho(Y)$, for any $\varepsilon>0$, $K=\mathrm{supp}_\varepsilon(a)$ is compact and hence $a[K]$ is a compact subset of $B$, as $a$ is continuous on $Y$.  By the upper semicontinuity of the norm, $a[K]$ is then norm-bounded and hence so too is $a$, i.e. $a\in\mathcal{S}^\rho_\infty$.  This shows that $\mathcal{C}_0^\rho(Y)$ is indeed a subspace of $\mathcal{S}^\rho_\infty$.

Now take any $c\in\mathrm{cl}_\infty(\mathcal{C}^\rho(Y))$.  To see that $c$ is also continuous at each $y\in Y$, we repeat the argument in \cite[Chapter II Proposition 13.12/Corollary 13.13]{DoranFell1988} but for general (potentially non-continuous) Banach bundles.  Specifically, take $(x_\lambda)\subseteq X$ with $x_\lambda\rightarrow y\in Y$.  As $\rho$ is an open map, we may revert to a subnet if necessary and assume that we also have $(b_\lambda)\subseteq B$ with $b_\lambda\rightarrow c(y)$ and $\rho(b_\lambda)=x_\lambda$, for all $\lambda$ (see \cite[Proposition 13.2]{DoranFell1988}).  For any $\delta>0$, we have $a\in\mathcal{C}^\rho(Y)$ with $\|a-c\|_\infty<\delta$.  In particular, $\|a(y)-c(y)\|<\delta$ and hence $\limsup\|a(x_\lambda)-b_\lambda\|<\delta$, as the norm is upper semicontinuous.  As $\|a(x_\lambda)-c(x_\lambda)\|<\delta$ too, for all $\lambda$, it follows that $\limsup\|c(x_\lambda)-b_\lambda\|<2\delta$.  As $\delta>0$ was arbitrary, $\|c(x_\lambda)-b_\lambda\|\rightarrow0$ and hence $c(x_\lambda)-b_\lambda\rightarrow 0_x$, by \eqref{TopologicalCompatibility} above.  As $b_\lambda\rightarrow c(y)$ and addition is continuous, $c(x_\lambda)\rightarrow c(y)$.  This shows $c$ is continuous at each $y\in Y$ and hence $c\in\mathcal{C}^\rho(Y)$, which in turn shows that $\mathrm{cl}_\infty(\mathcal{C}^\rho(Y))=\mathcal{C}^\rho(Y)$.

If $c\in\mathrm{cl}_\infty(\mathcal{C}_0^\rho(Y))$ then again, for any $\delta>0$, we have $a\in\mathcal{C}_0^\rho(Y)$ with $\|a-c\|_\infty<\delta$ and hence $\mathrm{supp}_{2\delta}(c)\subseteq\mathrm{supp}_{\delta}(a)\subseteq Y$.  As $c$ is continuous on $Y$ and the norm is upper semicontinuous, $\mathrm{supp}_{2\delta}(c)$ is then a closed subset of the compact set $\mathrm{supp}_{\delta}(a)$ and thus itself compact.  As $\delta>0$ was arbitrary, it follows $c$ vanishes at $\infty$, showing that $\mathrm{cl}_\infty(\mathcal{C}^\rho_0(Y))=\mathcal{C}^\rho_0(Y)$.

\item Note $f\cdot a\in\mathcal{C}_0^\rho(Y)$ when $a\in\mathcal{C}_0^\rho(Y)$ and $f:Y\rightarrow[0,1]$ is continuous, where
\begin{equation}\label{fdota}
f\cdot a(x)=\begin{cases}f(x)a(x)&\text{ if }x\in Y\\0_x&\text{if }x\in X\setminus Y.\end{cases}
\end{equation}
Indeed, if we have a net $(y_\lambda)\subseteq Y$ converging to $y\in Y$ then the continuity of $f$ and $a$ on $Y$ and the (joint) continuity of scalar multiplication yields $f\cdot a(y_\lambda)=f(y_\lambda)a(y_\lambda)\rightarrow f(y)a(y)$.  On the other hand, if $(x_\lambda)\subseteq X\setminus Y$ converges to $y\in Y$ then $f\cdot a(x_\lambda)=0_{x_\lambda}=f(y)a(x_\lambda)\rightarrow f(y)a(y)$, by the continuity of $a$ at $y$ and the continuity of multiplication by the scalar $f(y)$.  Thus $f\cdot a$ is continuous at each $y\in Y$ so $\mathrm{supp}_\delta(f\cdot a)$ is a closed and hence compact subset of $\mathrm{supp}_\delta(a)$, for all $\delta>0$, showing that $f\cdot a\in\mathcal{C}_0^\rho(Y)$.

Now if $Y$ is locally compact and Hausdorff, for any $a\in\mathcal{C}_0^\rho(Y)$ and $n\in\mathbb{N}$, Urysohn yields continuous $f_n:Y\rightarrow[0,1]$ and compact $K_n\subseteq Y$ with $\mathrm{supp}_{1/n}(a)\subseteq f_n^{-1}\{1\}$ and $\mathrm{supp}(f_n)\subseteq K_n$.  For each $n\in\mathbb{N}$, this also holds for the pointwise supremum $v_n=\bigvee_{k=1}^nf_k$.  Then $a_n=v_n\cdot a\in\mathcal{C}^\rho(Y)$ converges uniformly to $a$ and, for all $n\in\mathbb{N}$,
\[Y\supseteq\bigcup_{k=1}^nK_k\supseteq\mathrm{supp}(v_n)\supseteq\mathrm{supp}(a_n)\subseteq\mathrm{supp}(a_{n+1})\subseteq\mathrm{supp}(a).\]

\item When the norm is also continuous, $\mathrm{supp}_{1/n}(a)\subseteq\mathrm{int}(\mathrm{supp}_{1/(n+1)}(a))$ so we can choose each $f_n$ to have support in compact $K_n\subseteq\mathrm{supp}_{1/(n+1)}(a)$.  Then again $f_n\cdot a\in\mathcal{C}^\rho(Y)$ and $\mathrm{supp}(f_n\cdot a)=\mathrm{supp}(f_n)\subseteq K_n\subseteq\mathrm{supp}(f_{n+1})$. \qedhere
\end{enumerate}
\end{proof}

For any Hausdorff subspaces $Y,Z\subseteq X$, we see that
\[Y\subseteq Z\qquad\Rightarrow\qquad\mathcal{C}_0^\rho(Y)\subseteq\mathcal{C}_0^\rho(Z).\]
Indeed, if $a\in\mathcal{C}_0^\rho(Y)$ then certainly $a$ is continuous at each point of $Y$.  On the other hand, for each $\delta>0$, $\mathrm{supp}_\delta(a)$ is compact and hence closed in $Z$, as $Z$ is Hausdorff.  This shows that $a$ is also continuous at each point of $Z\setminus Y$ and hence $a\in\mathcal{C}_0^\rho(Z)$.

However, this can fail when $Z$ is not Hausdorff.  To `correct' this, we do the usual thing and take linear spans.  Noting $\mathcal{C}^\rho(K)=\mathcal{C}^\rho_0(K)$, for compact $K$, we define
\begin{equation}\label{Ch}
\mathcal{C}_\mathsf{h}^\rho=\mathrm{span}\{a\in\mathcal{C}^\rho(K):K\subseteq X\text{ is compact Hausdorff}\}.
\end{equation}
In general, elements of $\mathcal{C}_0^\rho(K)$ can be discontinuous at points of $\mathrm{cl}(K)\setminus K$ and so it is certainly possible that $\mathcal{C}_\mathsf{h}^\rho\nsubseteq\mathcal{C}^\rho_0$, where $\mathcal{C}^\rho_0=\mathcal{C}^\rho_0(X)$.  When $X$ is locally compact and Hausdorff, however, $\mathcal{C}_\mathsf{h}^\rho$ is precisely the uniform closure of $\mathcal{C}^\rho_0$.

\begin{prp}
For any Banach bundle $\rho:B\twoheadrightarrow X$ with locally compact Hausdorff base $X$,
\[\mathrm{cl}_\infty(\mathcal{C}_\mathsf{h}^\rho)=\mathcal{C}^\rho_0.\]
\end{prp}

\begin{proof}
If $X$ is Hausdorff then $\mathcal{C}_\mathsf{h}^\rho\subseteq\mathcal{C}^\rho_0$ so $\mathrm{cl}_\infty(\mathcal{C}_\mathsf{h}^\rho)\subseteq\mathcal{C}^\rho_0$, by \autoref{C0} \eqref{C01}.  If $X$ is also locally compact then $\mathcal{C}^\rho_0\subseteq\mathrm{cl}_\infty(\mathcal{C}_\mathsf{h}^\rho)$, by \autoref{C0} \eqref{C02}.
\end{proof}

Our primary interest will actually be in a slightly different analog of $\mathcal{C}_\mathsf{h}^\rho$ for Fell bundles over \'etale groupoids -- see $\mathcal{C}^\rho_\mathsf{c}$ defined below in \eqref{CcDef}.  To get an appropriate analog of $\mathcal{C}^\rho_0$, we will take the closure of $\mathcal{C}^\rho_\mathsf{c}$, again with respect to a different $\mathsf{b}$-norm which is better suited to the convolution product -- see $\mathcal{C}^\rho_\mathsf{r}$ defined below in \eqref{ReducedDefinition}.

\subsection{Fell Bundles}

We are primarily interested in Banach bundles where both the base and the total space have some compatible *-semigroupoid structure.

\begin{dfn}
A \emph{Fell bundle} is a Banach bundle $\rho:B\twoheadrightarrow\Gamma$ such that, moreover,
\begin{enumerate}
\item\label{Fell1} $B$ is a topological *-semigroupoid with bilinear products and antilinear *.
\item\label{Fell2} $\|\cdot\|$ is submultiplicative, $\Gamma$ is an \'etale groupoid and $\rho$ is a *-isocofibration.
\item\label{Fell3} For all $b\in B$, $\|b^*b\|=\|b\|^2$ and we have $a\in B_{\rho(b^*b)}$ with $a^*a=b^*b$.
\end{enumerate}
\end{dfn}

Here our primary reference is \cite{Kumjian1998}, although the Fell bundles in \cite{Kumjian1998} are again assumed to have continuous norm and locally compact Hausdorff base.  Even when the base is Hausdorff, the total space need not be, as in the following.

\begin{xpl}
Let $\rho:B\twoheadrightarrow\mathbb{N}\cup\{\infty\}$ be the Fell bundle whose base is the one-point compactification of $\mathbb{N}$ (considered as groupoid with trivial product) with $1$-dimensional fibres on $\mathbb{N}$ and a $2$-dimensional fibre at $\infty$, i.e.
\[B=(\mathbb{N}\times\mathbb{C})\cup(\{\infty\}\times(\mathbb{C}\oplus\mathbb{C})),\]
where $\mathbb{N}\times\mathbb{C}$ has the usual product topology but where, for any sequence $(z_n)\subseteq\mathbb{C}$,
\[(n,z_n)\rightarrow(\infty,(y,z))\qquad\Leftrightarrow\qquad z_n\rightarrow z.\]
In particular, $(n,0)\rightarrow(\infty,(z,0))$, for all $z\in\mathbb{C}$, so $B$ is certainly not Hausdorff.  Taking $z\neq0$, we also see that the norm is not continuous, only upper semicontinuous.  This is no accident, as the total space is automatically Hausdorff when the base is Hausdorff and the norm is continuous \textendash\, see \cite[Proposition 16.4]{Gierz1982}.
\end{xpl}

Conditions \eqref{Fell1} and \eqref{Fell2} essentially say that all the various structures on a Fell bundle are compatible.  Condition \eqref{Fell3} is needed to ensure that Fell bundles have a more C*-algebraic character.  Indeed, it follows from the C*-norm condition $\|b^*b\|=\|b\|^2$ that each unit fibre $B_x$, for $x\in\Gamma^0$, is a C*-algebra in its own right.  While this is not true for non-unit fibres, there are still certain C*-algebraic things we can say about arbitrary $b\in B$.  For example, submultiplicativity and the C*-norm condition yield $\|b\|^2=\|b^*b\|\leq\|b\|\|b^*\|$ so $\|b\|\leq\|b^*\|$ and hence $\|b^*\|\leq\|b^{**}\|=\|b\|$ too, i.e. the norm is $*$-invariant on the whole of $B$.  Similarly, we have the following result, which will be useful in the next section.

\begin{prp}\label{BelowNorm}
For any Fell bundle $\rho:B\twoheadrightarrow\Gamma$ and $(a,b)\in B^2$,
\begin{equation}\label{MidBelow}
b^*a^*ab\leq\|a\|^2b^*b.
\end{equation}
\end{prp}

\begin{proof}
Let $x=\mathsf{s}(\rho(a))=\mathsf{r}(\rho(b))$.  Note that if $c,d\in B_x$ and $c\leq d$ then $b^*cb\leq b^*db$, as $b^*(d-c)b=(b\sqrt{d-c})^*b\sqrt{d-c}\geq0$, by the last defining property of Fell bundles.  Next note that $B_{\rho(b)}$ is a left Hilbert $B_x$-module and hence $u_\lambda b\rightarrow b$, for any approximate unit $(u_\lambda)\subseteq B_x$.  As $a^*a$ lies in the self-adjoint part of $B_x$, it follows that $u_\lambda^*a^*au_\lambda\leq\|a^*a\|u_\lambda^*u_\lambda=\|a\|^2u_\lambda^*u_\lambda$ and hence
\[b^*a^*ab=\lim_\lambda b^*u_\lambda^*a^*au_\lambda b\leq\lim_\lambda\|a\|^2b^*u_\lambda^*u_\lambda b=\|a\|^2b^*b.\qedhere\]
\end{proof}

Just as some results only apply to unital C*-algebras, it is sometimes necessary to restrict to Fell bundles with a nice collection of units (e.g. when we want to identify multipliers of the reduced C*-algebra with `locally reduced' sections, as in \autoref{LeftLocallyReduced} below).  Specifically, let us call a Fell bundle $\rho:B\twoheadrightarrow\Gamma$ \emph{categorical} if $B$ is a topological category.  This means that each unit fibre contains a unit, i.e. $\rho[B^0]=\Gamma^0$, and, moreover, the source and range maps on $B$ are continuous.

\begin{prp}\label{CategoricalChars}
For any Fell bundle $\rho:B\twoheadrightarrow\Gamma$, the following are equivalent.
\begin{enumerate}
\item\label{TopCat} $\rho$ is categorical.
\item\label{Homeo} $\rho|_{B^0}$ is a homeomorphism onto $\Gamma^0$.
\item\label{CBundle} $\rho|_{\mathbb{C}B^0}$ is a trivial C*-bundle over $\Gamma^0$, i.e.
\begin{equation}\label{CB0}
\lambda b\mapsto(\lambda,\rho(b)),
\end{equation}
for $\lambda\in\mathbb{C}$ and $b\in B^0$, is an isomorphism from $\mathbb{C}B^0$ onto $\mathbb{C}\times\Gamma^0$.
\end{enumerate}
\end{prp}

\begin{proof}\
\begin{itemize}
\item[\eqref{TopCat}$\Rightarrow$\eqref{Homeo}] By the last defining property of a Banach bundle, the map $\gamma\mapsto0_\gamma$ is continuous on $\Gamma$.  If $B$ is a topological category then the function $s$ defined by $s(\gamma)=\mathsf{s}(0_\gamma)$ is continuous on $\Gamma$.  As $\rho$ is a functor, for any $x\in\Gamma^0$,
\[\qquad\qquad x=\rho(0_x)=\rho(0_x\mathsf{s}(0_x))=\rho(0_x)\rho(\mathsf{s}(0_x))=x\rho(\mathsf{s}(0_x))=\rho(\mathsf{s}(0_x))=\rho(s(x)).\]
Also $b=\mathsf{s}(0_{\rho(b)})=s(\rho(b))$, for any $b\in B^0$, as $\rho$ is an isocofibration.  Thus $s|_{\Gamma^0}=\rho|_{B^0}^{-1}$ and hence $\rho|_{B^0}$ is a homeomorphism onto $\Gamma^0$.

\item[\eqref{Homeo}$\Rightarrow$\eqref{TopCat}] If $\rho$ restricted to $B^0$ is a homeomorphism onto $\Gamma^0$ then its inverse is a continuous map from $\Gamma^0$ onto $B^0$.  For any $b\in B$, note $(\rho(b),\mathsf{s}(\rho(b)))\in\Gamma^2$ so $(b,\rho|_{B^0}^{-1}(\mathsf{s}(\rho(b))))\in B^2$, as $\rho$ is an isocofibration, i.e. $\mathsf{s}(b)=\rho|_{B^0}^{-1}(\mathsf{s}(\rho(b)))$.  Likewise, $\mathsf{r}(b)=\rho|_{B^0}^{-1}(\mathsf{r}(\rho(b)))$, showing that $B$ has source and range maps that are defined and continuous everywhere, i.e. $B$ is a topological category.

\item[\eqref{Homeo}$\Rightarrow$\eqref{CBundle}] We immediately see that \eqref{CB0} is an algebraic isomorphism from $\mathbb{C}B^0$ to $\mathbb{C}\times\Gamma^0$.  To see that \eqref{CB0} is also a homeomorphism, take nets $(\lambda_n)\subseteq\mathbb{C}$ and $(b_n)\subseteq B^0$.  If $\lambda_nb_n\rightarrow\lambda b$, for some $\lambda\in\mathbb{C}$ and $b\in B^0$, then
\[\rho(b_n)=\rho(\lambda_nb_n)\rightarrow\rho(\lambda b)=\rho(b).\]
If $\rho|_{B^0}$ is a homeomorphism onto $\Gamma^0$ then this implies $b_n\rightarrow b$ and hence $(\lambda_n-\lambda)b_n\rightarrow\lambda b-\lambda b=0_{\rho(b)}$.  As the norm is upper semicontinuous, $|\lambda_n-\lambda|=\|(\lambda_n-\lambda)b_n\|\rightarrow0$, i.e. $\lambda_n\rightarrow\lambda$ so $(\lambda_n,\rho(b_n))\rightarrow(\lambda,\rho(b))$, showing that \eqref{CB0} is continuous.  On the other hand, if $\lambda_n\rightarrow\lambda$ and $\rho(b_n)\rightarrow x\in\Gamma^0$ then $b_n\rightarrow\rho|_{B^0}^{-1}(x)$ and hence $\lambda_nb_n\rightarrow\lambda\rho|_{B^0}^{-1}(x)$, showing that the inverse map is also continuous, i.e. \eqref{CB0} is a homeomorphism.

\item[\eqref{CBundle}$\Rightarrow$\eqref{Homeo}] If \eqref{CB0} is an isomorphism then, in particular, $b\mapsto(1,\rho(b))$ and hence $b\mapsto\rho(b)$ is a homeomorphism on $B^0$. \qedhere
\end{itemize}
\end{proof}

Note $\rho[B^0]=\Gamma^0$ alone is not enough to ensure that $\rho$ is categorical.  For example, consider the subbundle $\rho:B\rightarrow[0,1]$ of the trivial Fell bundle of $2\times2$ matrices over the unit interval $[0,1]$ where the fibre at $1$ is the $1$-dimensional subspace of matrices with $0$ entries in all but the top left corner, i.e.
\[B=\{(x,m)\in[0,1]\times M_2:x=1\Rightarrow m\in\mathbb{C}e_{11}\}.\]
Then $B$ still has a unit in the fibre at $1$, namely $(1,e_{11})$, so $\rho[B^0]=[0,1]$.  But $(1,e_{11})\neq\lim_{x\rightarrow1}(x,e_{11}+e_{22})$ so $\rho$ restricted to $B^0$ is not a homeomorphism.

With Fell bundles, we can restrict to slices in definition of $\mathcal{C}_\mathsf{h}^\rho$ in \eqref{Ch}.

\begin{prp}
If $\rho:B\twoheadrightarrow\Gamma$ is a Fell bundle then
\begin{equation}\label{CcSlices}
\mathcal{C}_\mathsf{h}^\rho=\mathrm{span}\{a\in\mathcal{C}^\rho(K):K\subseteq\Gamma\text{ is a compact Hausdorff slice}\}.
\end{equation}
\end{prp}

\begin{proof}
Take $a\in\mathcal{C}^\rho(K)$, for some compact Hausdorff $K\subseteq\Gamma$.  As $\Gamma$ is \'etale, we can cover $K$ with open slices $O_1,\ldots,O_n\subseteq\Gamma$.  As $K$ is compact Hausdorff, we have a subordinate partition of unity on $K$, i.e. we have continuous $f_1,\cdots,f_n:K\rightarrow[0,1]$ and compact $K_1,\ldots,K_n$ with $\mathrm{supp}(f_m)\subseteq K_m\subseteq K\cap O_m$, for all $m\leq n$, and $\sum_{m\leq n}f_m(\gamma)=1$, for all $\gamma\in K$.  Setting $a_m=f_m\cdot a$, as in \eqref{fdota}, we see that $a_m\in\mathcal{C}^\rho(K_m)$ and hence
\[a=\sum_{m\leq n}a_m\in\mathrm{span}\{a\in\mathcal{C}^\rho(K):K\subseteq\Gamma\text{ is a compact Hausdorff slice}\}.\qedhere\]
\end{proof}

When $\Gamma^0$ is not Hausdorff, we will want to instead work with subspace given by
\begin{equation}\label{CcDef}
\mathcal{C}^\rho_\mathsf{c}=\mathrm{span}\{a\in\mathcal{C}^\rho(K):K\subseteq\Gamma\text{ is bicompact}\}.
\end{equation}

\begin{prp}
If $\rho:B\twoheadrightarrow\Gamma$ is a Fell bundle and $\Gamma^0$ is Hausdorff then
\[\mathcal{C}_\mathsf{c}^\rho=\mathcal{C}_\mathsf{h}^\rho=\mathrm{span}\{a\in\mathcal{C}^\rho(K):K\subseteq\Gamma\text{ is a compact slice}\}.\]
\end{prp}

\begin{proof}
This follows from \eqref{CcSlices} and the fact that, when $\Gamma^0$ is Hausdorff, every compact slice is Hausdorff and contained in an open slice, by \cite[Proposition 6.3]{BiceStarling2018}.
\end{proof}

\begin{prp}\label{ComeagreContinuity}
For any Fell bundle $\rho:B\twoheadrightarrow\Gamma$,
\[a\in\mathrm{cl}_\infty(\mathcal{C}_\mathsf{c}^\rho)\qquad\Rightarrow\qquad\mathcal{C}(a)\text{ is comeagre}.\]
\end{prp}

\begin{proof}
Any bicompact $K$ is, by defintion, contained in some open slice $O$.  As $K$ is also biclosed, any $a\in\mathcal{C}^\rho(K)$ is necessarily continuous on $O$, which implies $\Gamma\setminus\mathcal{C}(a)\subseteq\mathrm{cl}(O)\setminus O$ is nowhere dense.  Finite unions of nowhere dense sets are again nowhere dense so the same applies to any $a\in\mathcal{C}^\rho_\mathsf{c}$.  So if $(a_n)\subseteq\mathcal{C}^\rho_\mathsf{c}$ and $\|a-a_n\|_\infty\rightarrow a$ then $\mathcal{C}(a)\subseteq\bigcap\mathcal{C}(a_n)$ is comeagre, by (the proof of) \autoref{C0} \eqref{C01}.
\end{proof}

\section{Products and Norms}

To avoid repeating our basic hypotheses we make the following assumption. 
\begin{center}
\textbf{Throughout the rest of this section $\rho:B\twoheadrightarrow\Gamma$ is a Fell bundle.}
\end{center}
Here we mainly focus on the norm-algebraic structure of $\rho$, ignoring the topology.  Recall that $\mathcal{S}^\rho$ denotes the sections of $\rho$.
For any $a\in\mathcal{S}^\rho$, we define $a^*\in\mathcal{S}^\rho$ by
\[a^*(\gamma)=a(\gamma^{-1})^*.\]
For any $\Delta\subseteq\Gamma$, we also define the \emph{$0$-restriction} $a_\Delta\in\mathcal{S}^\rho$ by
\[a_\Delta(\gamma)=\begin{cases}a(\gamma)&\text{if }\gamma\in\Delta\\0_\gamma&\text{if }\gamma\notin\Delta.\end{cases}\]
So $a_\Delta$ is the same as the usual restriction $a|_\Delta$ but with $0$ values outside $\Delta$.  In particular, we let $\Phi$ denote the $0$-restriction of any $a\in\mathcal{S}^\rho$ to the diagonal, i.e.
\[\Phi(a)=a_{\Gamma^0}.\]

\subsection{Convolution}

Let us denote finite subsets by $\subset$, i.e.
\[F\subset\Lambda\qquad\Leftrightarrow\qquad F\subseteq\Lambda\text{ and }|F|<\infty\]
(which could also be viewed as compact containment w.r.t. the discrete topology).  Note $\{F:F\subset\Lambda\}$ is directed by inclusion $\subseteq$.  If $V$ is a normed space, $(v_\lambda)_{\lambda\in\Lambda}\subseteq V$ and the net of finite partial sums $(\sum_{\lambda\in F}v_\lambda)_{F\subset\Lambda}$ is norm-convergent then we denote the limit by $\sum_{\lambda\in\Lambda}v_\lambda$.  Note this is equivalent to unconditional convergence, i.e. there are only countably many non-zero $v_\lambda$'s and any enumeration of them has the same (well-defined/convergent) sum.

Whenever possible, we define the convolution product of $a,b\in\mathcal{S}^\rho$ by
\[ab(\gamma)=\sum_{\gamma=\alpha\beta}a(\alpha)b(\beta),\]
i.e. whenever the finite partial sums $ab_F(\gamma)=\sum_{\beta\in F}a(\gamma\beta^{-1})b(\beta)$, for $F\subset\Gamma\gamma$, converge in each fibre $B_\gamma=\rho^{-1}\{\gamma\}$.  Note $ab_\Delta$ denotes the product of $a$ with $b_\Delta$, while $(ab)_\Delta$ denotes the $0$-restriction of $ab$ to $\Delta$ (if these products are defined).

\begin{prp}
Take any $F\subset\Gamma$ and $a,b,c\in\mathcal{S}^\rho$.  If $ab$ is defined then
\begin{equation}\label{abcF}
(ab)c_F=a(bc_F).
\end{equation}
Likewise, if $bc$ is defined then $(a_Fb)c=a_F(bc)$.
\end{prp}

\begin{proof}
As $F$ is finite and $ab$ is defined, so is $(ab)c_F$.  Specifically, for any $\gamma\in\Gamma$,
\begin{align*}
((ab)c_F)(\gamma)&=\sum_{\beta\in F\cap\Gamma\gamma}ab(\gamma\beta^{-1})c(\beta)\\
&=\sum_{\beta\in F\cap\Gamma\gamma}\lim_{G\subset\gamma\Gamma}\sum_{\alpha\in G}a(\alpha)b(\alpha^{-1}\gamma\beta^{-1})c(\beta)\\
&=\lim_{G\subset\gamma\Gamma}\sum_{\beta\in F\cap\Gamma\gamma}\sum_{\alpha\in G}a(\alpha)b(\alpha^{-1}\gamma\beta^{-1})c(\beta)\\
&=\lim_{G\subset\gamma\Gamma}\sum_{\alpha\in G}a(\alpha)\sum_{\beta\in F\cap\Gamma\gamma}b(\alpha^{-1}\gamma\beta^{-1})c(\beta)\\
&=\lim_{G\subset\gamma\Gamma}\sum_{\alpha\in G}a(\alpha)bc_F(\alpha^{-1}\gamma)\\
&=(a(bc_F))(\gamma)
\end{align*}
This shows that $a(bc_F)$ is also defined and equal to $(ab)c_F$ everywhere on $\Gamma$.  The second statement follows by a dual argument.
\end{proof}

Let us denote the finitely supported sections by
\[\mathcal{F}^\rho=\{f\in\mathcal{S}^\rho:|\mathrm{supp}(f)|<\infty\}.\]
The Cauchy-Schwarz inequality (see \cite[Lemma 2.5]{RaeburnWilliams1998}) yields the following.

\begin{lem}
For any $f,g\in\mathcal{F}^\rho$ and $\gamma\in\Gamma$.
\begin{equation}\label{gammaCS}
fg(\gamma)^*fg(\gamma)\leq\|fg(\gamma)\|\sqrt{\|ff^*(\mathsf{r}(\gamma))\|g^*g(\mathsf{s}(\gamma))},
\end{equation}
where $\leq$ denotes the usual ordering in the C*-algebra $B_{\mathsf{s}(\gamma)}$.
\end{lem}

\begin{proof}
Cauchy-Schwarz says that, for any $j$ and $k$ in a Hilbert module $H$,
\[\langle j,k\rangle^*\langle j,k\rangle\leq\|\langle j,j\rangle\|\langle k,k\rangle.\]
As $a^2\leq b^2$ implies $a\leq b$, for any positive $a$ and $b$ in a C*-algebra (see \cite[II.3.1.10]{Blackadar2017}), it follows that, whenever $\langle j,k\rangle$ is positive,
\begin{equation}\label{CSsqrt}
\langle j,k\rangle\leq\sqrt{\|\langle j,j\rangle\|\langle k,k\rangle}.
\end{equation}
In particular, we can consider the Hilbert $B_{\mathsf{s}(\gamma)}$-module
\[H=\{h\in\mathcal{S}^\rho:\mathrm{supp}(h)\subseteq\mathrm{supp}(g)\cap\Gamma\gamma\}\]
where $\langle j,k\rangle=j^*k(\mathsf{s}(\gamma))$.  By \eqref{CSsqrt} applied in $H$,
\begin{align*}
fg(\gamma)^*fg(\gamma)&=fg(\gamma)^*\sum_{\alpha\beta=\gamma}f(\alpha)g(\beta).\\
&=\sum_{\alpha\beta=\gamma}fg(\gamma)^*f(\alpha)g(\beta).\\
&\leq\sqrt{\Big\|\sum_{\alpha\in\gamma\Gamma}fg(\gamma)^*f(\alpha)f(\alpha)^*fg(\gamma)\Big\|\sum_{\beta\in\Gamma\gamma}g(\beta)^*g(\beta)}.\\
&\leq\|fg(\gamma)\|\sqrt{\|ff^*(\mathsf{r}(\gamma))\|g^*g(\mathsf{s}(\gamma))}.\qedhere
\end{align*}
\end{proof}

Actually, what we really need is the norm inequality resulting from \eqref{gammaCS}, i.e.
\begin{equation}\label{gammaCSnorm}
\|fg(\gamma)\|\leq\sqrt{\|ff^*(\mathsf{r}(\gamma))\|\|g^*g(\mathsf{s}(\gamma))\|}.
\end{equation}

\subsection{The $2$-Norm}

The $[0,\infty]$-valued $2$-norm on $\mathcal{S}^\rho$ is given by
\[\|a\|_2=\sup_{x\in G^0}\sup_{F\subset\Gamma x}\sqrt{\Big\|\sum_{\gamma\in F}a(\gamma)^*a(\gamma)\Big\|}.\]
Alternatively, this can be expressed more concisely as
\begin{equation}\label{Concise2Defintion}
\|a\|_2=\sup_{F\subset\Gamma}\sqrt{\|\Phi(a^*a_F)\|}_\infty=\sup_{F\subset\Gamma}\|a_F\|_2.
\end{equation}
On the sections forming the Hilbert module associated to $\rho$ in \autoref{TheBigHilbertModule} below, the $2$-norm will be none other than the canonical Hilbert norm.  However, the $2$-norm can still be finite for sections lying outside this module as well.

First note that the $2$-norm dominates the $\infty$-norm, i.e. for all $a\in\mathcal{S}^\rho$,
\begin{equation}\label{InfinityBelow2}
\|a\|_\infty\leq\|a\|_2.
\end{equation}
Indeed, for all $\gamma\in\Gamma$, the C*-condition on Fell bundles yields
\[\|a(\gamma)\|^2=\|a(\gamma)^*a(\gamma)\|=\|\Phi(a^*a_\gamma)\|_\infty\leq\|a\|_2^2\]
so taking suprema yields \eqref{InfinityBelow2}.  This extends to products as follows.

\begin{prp}
For any $a,b\in\mathcal{S}^\rho$ such that $a^*b$ is defined,
\begin{equation}\label{infty22}
\|a^*b\|_\infty\leq\|a\|_2\|b\|_2.
\end{equation}
\end{prp}

\begin{proof}
For any $\gamma\in\Gamma$ and $F\subset\Gamma\gamma$, \eqref{gammaCSnorm} applied to $a^*b_F(\gamma)=(a_{F\gamma^{-1}})^*b_F(\gamma)$ yields
\[\|a^*b_F(\gamma)\|\leq\sqrt{\|a^*a_{F\gamma^{-1}}(\mathsf{r}(\gamma))\|\|b^*b_F(\mathsf{s}(\gamma))\|}\leq\|a\|_2\|b\|_2.\]
Thus $\|a^*b(\gamma)\|=\lim_{F\subset\Gamma\gamma}\|a^*b_F(\gamma)\|\leq\|a\|_2\|b\|_2$, for all $\gamma\in\Gamma$, proving \eqref{infty22}.
\end{proof}

In particular, $\|a^*a_F\|_\infty\leq\|a\|_2\|a_F\|_2\leq\|a\|_2^2$.  It follows that the $\Phi$ in \eqref{Concise2Defintion} is actually unnecessary, i.e.
\[\|a\|_2=\sup_{F\subset\Gamma}\sqrt{\|a^*a_F\|}_\infty.\]

Now we can show that the $2$-norm is indeed a ($[0,\infty]$-valued) norm.

\begin{cor}\label{2subadditive}
The $2$-norm is subadditive.
\end{cor}

\begin{proof}
Certainly $\|\cdot\|_\infty$ is subadditive, as the norm in each fibre $B_\gamma$ is subadditive.  For any $a,b\in S^\rho$ and $F\subset\Gamma$, \eqref{infty22} then yields
\begin{align*}
\|(a+b)_F\|_2^2&=\|(a_F+b_F)^*(a_F+b_F)\|_\infty\\
&\leq\|a_F^*a_F\|_\infty+\|a_F^*b_F\|_\infty+\|b_F^*a_F\|_\infty+\|b_F^*b_F\|_\infty\\
&\leq\|a_F\|_2^2+2\|a_F\|_2\|b_F\|_2+\|b_F\|_2^2\\
&=(\|a_F\|_2+\|b_F\|_2)^2.
\end{align*}
Taking suprema yields $\|a+b\|_2\leq\|a\|_2+\|b\|_2$.
\end{proof}

We can also show that the $2$-norm is dominated by the $1$-norm defined by
\begin{equation}\label{1Norm}
\|a\|_1=\sup_{x\in\Gamma^0}\sum_{\gamma\in\Gamma x}\|a(\gamma)\|.
\end{equation}

\begin{cor}\label{21}
For all $a\in\mathcal{S}^\rho$,
\begin{equation}\label{2below1}
\|a\|_2\leq\|a\|_1
\end{equation}
If $\mathsf{s}$ is injective on $\mathrm{supp}(a)$ then $\|a\|_1=\|a\|_2=\|a\|_\infty$.
\end{cor}

\begin{proof}
For any $x\in\Gamma^0$ and $F\subset\Gamma x$, subadditivity yields
\[\|a_F\|_2\leq\sum_{\gamma\in F}\|a_\gamma\|_2=\sum_{\gamma\in F}\|a(\gamma)\|=\|a_F\|_1\leq\|a\|_1.\]
Taking suprema over thus yields \eqref{2below1}.  If $\mathsf{s}$ is injective on $\mathrm{supp}(a)$ then there can be at most one non-zero term in the above sums so $\|a_F\|_1=\sum_{\gamma\in F}\|a(\gamma)\|\leq\|a\|_\infty$.  Taking suprema then yields $\|a\|_1\leq\|a\|_\infty$ and hence $\|a\|_1=\|a\|_2=\|a\|_\infty$, by \eqref{InfinityBelow2} and the inequality \eqref{2below1} just proved.
\end{proof}

One situation when $ab$ is necessarily defined is when $\mathsf{r}$ is injective on $\mathrm{supp}(a)$ or $\mathsf{s}$ is injective on $\mathrm{supp}(b)$, in which case we can bound the $2$-norm of $ab$ as follows.

\begin{prp}
If $a,b\in\mathcal{S}^\rho$ and $\mathsf{s}$ is injective on $\mathrm{supp}(b)$ then $ab$ is defined and
\begin{equation}\label{2Infinity}
\|ab\|_2\leq\|a\|_2\|b\|_\infty.
\end{equation}
On the other hand, if $\mathrm{supp}(a)$ is a slice then $ab$ is again defined and
\begin{equation}\label{Infinity2}
\|ab\|_2\leq\|a\|_\infty\|b\|_2.
\end{equation}
\end{prp}

\begin{proof}
Take any $\beta\in\mathsf{s}[\mathrm{supp}(b)]$.  If $\mathsf{s}$ is injective on $\mathrm{supp}(b)$ then $ab(\gamma)=a(\gamma\beta^{-1})b(\beta)$, for any $\gamma\in\Gamma\beta$.  For any $F\subset\Gamma\beta$, it follows that
\[\sum_{\gamma\in F}ab(\gamma)^*ab(\gamma)=b(\beta)^*\Big(\sum_{\gamma\in F}a(\gamma\beta^{-1})^*a(\gamma\beta^{-1})\Big)b(\beta)=b(\beta)^*(a^*a_{F\beta^{-1}})(\mathsf{s}(\beta))b(\beta).\]
Taking norms yields
\[\|((ab)^*(ab)_F)(\mathsf{s}(\beta))\|\leq\|a^*a_{F\beta^{-1}}(\mathsf{s}(\beta))\|\|b(\beta)\|^2\leq\|a\|_2^2\|b\|_\infty^2.\]
On the other hand, if $x\in\Gamma^0\setminus\mathsf{s}[\mathrm{supp}(b)]$ then $((ab)^*(ab))(x)=0_x$ and hence this yields the required inequality $\|ab\|_2\leq\|a\|_2\|b\|_\infty$.

Similarly, if $\mathsf{r}$ is injective on $\mathrm{supp}(a)$ then, for any $\gamma\in\mathsf{r}^{-1}[\mathsf{r}[\mathrm{supp}(a)]]$, there is precisely one $\alpha_\gamma\in\mathrm{supp}(a)$ with $\mathsf{r}(\alpha_\gamma)=\mathsf{r}(\gamma)$ and hence $ab(\gamma)=a(\alpha_\gamma)b(\alpha_\gamma^{-1}\gamma)$.  If $\mathsf{s}$ is also injective on $\mathrm{supp}(a)$ then the $\alpha_\gamma^{-1}\gamma$'s here are distinct for distinct $\gamma$'s.  Indeed, $\alpha_\beta^{-1}\beta=\alpha_\gamma^{-1}\gamma$ implies $\mathsf{s}(\alpha_\beta)=\mathsf{r}(\alpha_\beta^{-1}\beta)=\mathsf{r}(\alpha_\gamma^{-1}\gamma)=\mathsf{s}(\alpha_\gamma)$ and hence $\alpha_\gamma=\alpha_\beta$ so $\beta=\alpha_\beta\alpha_\beta^{-1}\beta=\alpha_\gamma\alpha_\gamma^{-1}\gamma=\gamma$.  Now, for any $x\in\Gamma^0$ and $F\subset\Gamma x\cap\mathsf{r}^{-1}[\mathsf{r}[\mathrm{supp}(a)]]$,
\begin{align*}
(ab)^*(ab)_F(x)&=\sum_{\gamma\in F}ab(\gamma)^*ab(\gamma)\\
&=\sum_{\gamma\in F}b(\alpha_\gamma^{-1}\gamma)^*a(\alpha_\gamma)^*a(\alpha_\gamma)b(\alpha_\gamma^{-1}\gamma)\\
&\leq\sum_{\gamma\in F}\|a(\alpha_\gamma)\|^2b(\alpha_\gamma^{-1}\gamma)^*b(\alpha_\gamma^{-1}\gamma),\text{ by \eqref{MidBelow}},\\
&\leq\|a\|_\infty^2(b^*b_G)(x),
\end{align*}
where $G=\{\alpha_\gamma^{-1}\gamma:\gamma\in F\}\subset\Gamma\gamma$, as the $\alpha_\gamma^{-1}\gamma$'s are distinct.  Taking norms then yields $\|(ab)^*(ab)_F(x))\|\leq\|a\|_\infty^2\|(b^*b_G)(x)\|\leq\|a\|_\infty^2\|b\|_2^2$.  For any $F\subset\Gamma x$, we immediately see that $(ab)_F=(ab)_{F\cap\mathsf{r}^{-1}[\mathsf{r}[\mathrm{supp}(a)]]}$ so taking suprema over $F\subset\Gamma x$ yields the other required inequality $\|a\|_2\leq\|a\|_\infty\|b\|_2$.
\end{proof}

Let us denote the sections with finite $2$-norm by
\[\mathcal{S}_2^\rho=\{a\in\mathcal{S}^\rho:\|a\|_2<\infty\}.\]

\begin{prp}\label{2Banach}
$(\mathcal{S}_2^\rho,\|\cdot\|_2)$ is a Banach space.
\end{prp}

\begin{proof}
We just need to show that $\mathcal{S}_2^\rho$ is 2-complete.  Accordingly, take any $2$-Cauchy $(a_n)\subseteq\mathcal{S}_2^\rho$.  By \eqref{infty2b}, $(a_n)$ is $\infty$-Cauchy and hence has an $\infty$-limit $a\in\mathcal{S}_\infty^\rho$, i.e. $\|a-a_n\|_\infty\rightarrow0$.  For any $F\subset\Gamma$ of size $|F|\in\mathbb{N}$, \autoref{21} then yields
\[\|(a-a_m)_F\|_2\leq\|(a-a_m)_F\|_1\leq|F|\|a-a_m\|_\infty\rightarrow0\]
 and hence
\[\|(a-a_n)_F\|_2=\lim_{m\rightarrow\infty}\|(a_m-a_n)_F\|_2\leq\lim_{m\rightarrow\infty}\|a_m-a_n\|_2.\]
Taking the supremum, it follows that $\|a-a_n\|_2\leq\lim_{m\rightarrow\infty}\|a_m-a_n\|_2$ and hence $\lim_{n\rightarrow\infty}\|a-a_n\|_2=0$, as $(a_n)$ is $2$-Cauchy.  In particular, $\|a-a_n\|_2<\infty$, for some $n$, so $\|a\|_2\leq\|a-a_n\|_2+\|a_n\|_2<\infty$, by \autoref{2subadditive}.  Thus $a\in\mathcal{S}_2^\rho$ is a $2$-limit of $(a_n)$, showing $\mathcal{S}_2^\rho$ is $2$-complete and hence a Banach space w.r.t. $\|\cdot\|_2$.
\end{proof}

Important subspaces of $\mathcal{S}^\rho_2$ are formed from sections $a$ for which the *-square $a^*a$ is defined.  We can characterise such sections as follows.

\begin{lem}\label{a*aDefined}
If $a\in\mathcal{S}^\rho$ then $a^*a$ is defined precisely when, for each $x\in\Gamma^0$,
\begin{equation}\label{a*a}
\lim_{F\subset\Gamma x}\|a_{\Gamma x\setminus F}\|_2=0,
\end{equation}
in which case $\|a\|_2=\sqrt{\|a^*a\|_\infty}=\sqrt{\|\Phi(a^*a)\|_\infty}$.
\end{lem}

\begin{proof}
If $a^*a$ is defined then, for each $x\in\Gamma^0$,
\[\lim_{F\subset\Gamma x}\|a_{\Gamma x\setminus F}\|_2=\lim_{F\subset\Gamma x}\|(a^*a-a^*a_F)(x)\|=0,\]
i.e. \eqref{a*a} holds.  Conversely, say \eqref{a*a} holds, for all $x\in X$, and take $\gamma\in\Gamma$.  Whenever $F\subseteq G\subset\Gamma\gamma$, \eqref{infty22} yields
\[\|(a^*a_G-a^*a_F)(\gamma)\|=\|a^*a_{G\setminus F}(\gamma)\|\leq\|a_{(G\setminus F)\gamma^{-1}}\|_2\|a_{G\setminus F}\|_2\leq\|a_{\Gamma\gamma^{-1}}\|_2\|a_{\Gamma\gamma\setminus F}\|_2.\]
As $\lim_{F\subset\Gamma x}\|a_{\Gamma x\setminus F}\|_2=0$, for $x=\mathsf{s}(\gamma)$ and $\mathsf{r}(\gamma)$, this shows that $(a^*a_F(\gamma))_{F\subset\Gamma\gamma}$ is Cauchy and hence converges, as $B_\gamma$ is a Banach space.  As $\gamma$ was arbitrary, this shows that $a^*a$ is defined.

Now if $a^*a$ is defined then, in particular, $(\|a^*a_F(x)\|)_{F\subset\Gamma x}$ increases to $\|a^*a(x)\|$, for each $x\in\Gamma$.  Together with \eqref{infty22} again, this implies
\[\|a\|_2^2=\sup_{F\subset\Gamma}\|\Phi(a^*a_F)\|_\infty=\|\Phi(a^*a)\|_\infty\leq\|a^*a\|_\infty\leq\|a\|_2^2.\qedhere\]
\end{proof}

Using this, we can relate the existence of *-squares with general products.

\begin{prp}\label{a*b}
If $a\in\mathcal{S}_2^\rho$ and $b^*b$ is defined then so is $a^*b$.
\end{prp}

\begin{proof}
For all $\gamma\in\Gamma$ and $F\subseteq G\subset\Gamma\gamma$, \eqref{infty22} yields
\[\|a^*b_G(\gamma)-a^*b_F(\gamma)\|=\|a^*b_{G\setminus F}(\gamma)\|\leq\|a_{(G\setminus F)\gamma^{-1}}\|_2\|b_{G\setminus F}\|_2\leq\|a\|_2\|b_{\Gamma\gamma\setminus F}\|_2.\]
It follows that $(a^*b_F(\gamma))_{F\subset\Gamma\gamma}$ is Cauchy, as $\lim_{F\subset\Gamma\gamma}\|b_{\Gamma\gamma\setminus F}\|_2=0$, by \eqref{a*a}.  As $B_\gamma$ is a Banach space, it follows that $a^*b(\gamma)$ is defined, for all $\gamma\in\Gamma$.
\end{proof}

\subsection{The $\mathsf{b}$-Norm}

Next we consider another $[0,\infty]$-valued norm on $\mathcal{S}^\rho$ coming from the left regular representation on finitely supported sections.  Specifically, let
\[\|a\|_\mathsf{b}=\sup_{f\in\mathcal{F}^\rho}\frac{\|af\|_2}{\|f\|_2}\]
(taking $0/0=0$).  On the sections forming the reduced C*-algebra, the $\mathsf{b}$-norm will be none other than the usual reduced norm.  But again the $\mathsf{b}$-norm will be finite on many sections lying outside the reduced C*-algebra as well.

\begin{prp}\label{infty2bprp}
For all $a\in\mathcal{S}^\rho$,
\begin{equation}\label{infty2b}
\|a\|_2\leq\|a\|_\mathsf{b}=\|a^*\|_\mathsf{b}.
\end{equation}
If $\mathrm{supp}(a)$ is a slice of $\Gamma$ then $\|a\|_2=\|a\|_\mathsf{b}$.
\end{prp}

\begin{proof}
For any $F\subset\Gamma$,
\[\|a^*a_F\|_\infty\leq\|a^*a_F\|_2\leq\|a^*\|_\mathsf{b}\|a_F\|_2\leq\|a^*\|_\mathsf{b}\|a\|_2.\]
Taking suprema, we get $\|a\|_2^2\leq\|a^*\|_\mathsf{b}\|a\|_2$ and hence $\|a\|_2\leq\|a^*\|_\mathsf{b}$.

Now let us temporarily define another $[0,\infty]$-valued norm on $\mathcal{S}^\rho$ by
\[\|a\|'_\mathsf{b}=\sup_{f,g\in\mathcal{F}^\rho}\frac{\|\Phi(g^*af)\|_\infty}{\|g\|_2\|f\|_2}.\]
We claim that $\|a\|'_\mathsf{b}=\|a\|_\mathsf{b}$, for all $a\in\mathcal{S}^\rho$.  To see this note that, for any $f\in\mathcal{F}^\rho$,
\[\|af\|_2^2=\sup_{F\subset\Gamma}\|\Phi((af)^*_Faf)\|_\infty\leq\sup_{F\subset\Gamma}\|(af)_F\|_2\|a\|'_\mathsf{b}\|f\|_2=\|af\|_2\|a\|'_\mathsf{b}\|f\|_2.\]
Thus $\|af\|_2\leq\|a\|'_\mathsf{b}\|f\|_2$ and hence $\|a\|_\mathsf{b}\leq\|a\|'_\mathsf{b}$.  Conversely, for any $f,g\in\mathcal{F}^\rho$,
\[\|\Phi(g^*af)\|_\infty\leq\|g^*af\|_\infty\leq\|g\|_2\|af\|_2\leq\|g\|_2\|a\|_\mathsf{b}\|f\|_2,\]
by \eqref{infty22}, and hence $\|a\|'_\mathsf{b}\leq\|a\|_\mathsf{b}$, proving the claim.

For all $a\in\mathcal{S}^\rho$, $\|a^*\|_\infty=\|a\|_\infty$ (as $\|b\|=\|b^*\|$, for all $b\in B$) and hence
\[\|a^*\|_\mathsf{b}=\sup_{f,g\in\mathcal{F}^\rho}\frac{\|\Phi(g^*a^*f)\|_\infty}{\|g\|_2\|f\|_2}=\sup_{g,f\in\mathcal{F}^\rho}\frac{\|\Phi(f^*ag)^*\|_\infty}{\|f\|_2\|g\|_2}=\|a\|_\mathsf{b},\]
which completes the proof of \eqref{infty2b}.

Lastly note that if $\mathrm{supp}(a)$ is a slice then \eqref{Infinity2} yields $\|af\|_2\leq\|a\|_\infty\|f\|_2$, for all $f\in\mathcal{F}^\rho$.  Thus $\|a\|_\mathsf{b}\leq\|a\|_\infty\leq\|a\|_2$, by \eqref{InfinityBelow2}, so $\|a\|_2=\|a\|_\mathsf{b}$.
\end{proof}

The proof above actually showed that the $\Phi$ in the definition of $\|\cdot\|'_\mathsf{b}$ is again unnecessary.  From now on we will freely use these characterisations of $\|\cdot\|_\mathsf{b}$, i.e.
\begin{equation}\label{b'}
\|a\|_\mathsf{b}=\sup_{f,g\in\mathcal{F}^\rho}\frac{\|g^*af\|_\infty}{\|g\|_2\|f\|_2}=\sup_{f,g\in\mathcal{F}^\rho}\frac{\|\Phi(g^*af)\|_\infty}{\|g\|_2\|f\|_2},
\end{equation}
for example in the following proof that the $\mathsf{b}$-norm is submultiplicative.

\begin{prp}
For any $a,b\in\mathcal{S}^\rho$ such that $ab$ is defined,
\begin{equation}\label{||ab||_b}
\|ab\|_\mathsf{b}\leq\|a\|_\mathsf{b}\|b\|_\mathsf{b}.
\end{equation}
\end{prp}

\begin{proof}
Simply note that, for all $f,g\in\mathcal{F}^\rho$, \eqref{infty22} and \eqref{infty2b} yield
\[\|g^*abf\|_\infty\leq\|a^*g\|_2\|bf\|_2\leq\|a\|_\mathsf{b}\|g\|_2\|b\|_\mathsf{b}\|f\|_2\]
(here we are also implicitly using \eqref{abcF}) and hence $\|ab\|_\mathsf{b}\leq\|a\|_\mathsf{b}\|b\|_\mathsf{b}$, by \eqref{b'}.
\end{proof}

We can now show that the $\mathsf{b}$-norm satisfies the C*-norm condition.

\begin{prp}
If $a\in\mathcal{S}^\rho$ and $a^*a$ is defined then
\[\|a^*a\|_\mathsf{b}=\|a\|_\mathsf{b}^2.\]
\end{prp}

\begin{proof}
For any $f\in\mathcal{F}^\rho$, \autoref{a*aDefined} yields $\|af\|_2^2=\|f^*a^*af\|_\infty\leq\|a^*a\|_\mathsf{b}\|f\|_2^2$ and hence $\|a\|_\mathsf{b}^2\leq\|a^*a\|_\mathsf{b}$.  On the other hand, $\|a^*a\|_\mathsf{b}\leq\|a^*\|_\mathsf{b}\|a\|_\mathsf{b}=\|a\|_\mathsf{b}^2$, by \eqref{infty2b} and \eqref{||ab||_b}, and hence $\|a\|_\mathsf{b}^2=\|a^*a\|_\mathsf{b}$.
\end{proof}

Again we denote the sections with finite $\mathsf{b}$-norm by
\[\mathcal{S}_\mathsf{b}^\rho=\{a\in\mathcal{S}^\rho:\|a\|_\mathsf{b}<\infty\}.\]

\begin{prp}\label{bBanach}
$(\mathcal{S}_\mathsf{b}^\rho,\|\cdot\|_\mathsf{b})$ is a Banach space.
\end{prp}

\begin{proof}
If $(a_n)\subseteq\mathcal{S}_\mathsf{b}^\rho$ is $\mathsf{b}$-Cauchy and hence $2$-Cauchy then it has a $2$-limit $a\in\mathcal{S}_2^\rho$, by \autoref{2Banach}.  For any $f\in\mathcal{F}^\rho$, \autoref{2subadditive} and \eqref{2Infinity} yield
\[\|(a-a_m)f\|_2\leq\sum_{\gamma\in\mathrm{supp}(f)}\|a-a_m\|_2\|f(\gamma)\|\rightarrow0\]
and hence $\|(a-a_n)f\|_2=\lim_{m\rightarrow\infty}\|(a_m-a_n)f\|_2\leq\lim_{m\rightarrow\infty}\|a_m-a_n\|_\mathsf{b}\|f\|_2$.  Thus $\|a-a_n\|_\mathsf{b}\leq\lim_{m\rightarrow\infty}\|a_m-a_n\|_\mathsf{b}$ and hence $\lim_{n\rightarrow\infty}\|a-a_n\|_\mathsf{b}=0$, as $(a_n)$ is $\mathsf{b}$-Cauchy.  By subadditivity of $\|\cdot\|_\mathsf{b}$, it follows that $a\in\mathcal{S}_\mathsf{b}^\rho$ is a $\mathsf{b}$-limit of $(a_n)$, showing that $\mathcal{S}_\mathsf{b}^\rho$ is $\mathsf{b}$-complete and hence a Banach space w.r.t. $\|\cdot\|_\mathsf{b}$.
\end{proof}

\begin{xpl}\label{NN}
If $\Gamma=\mathbb{N}\times\mathbb{N}$ is the discrete groupoid arising from full equivalence relation on $\mathbb{N}$, and $\rho:\mathbb{C}\times\Gamma\twoheadrightarrow\Gamma$ is the trivial complex line bundle over $\Gamma$ then every element of $\mathcal{S}^\rho$ is an $\mathbb{N}\times\mathbb{N}$ matrix with entries in $\mathbb{C}$.  In this case, $\mathcal{S}_\infty^\rho$ consists of the matrices whose entries form a bounded subset of $\mathbb{C}$, $\mathcal{S}_2^\rho$ consists of matrices whose columns form a bounded subset of the Hilbert space $\ell^2$, while $\mathcal{S}_\mathsf{b}^\rho$ consists of matrices which correspond to bounded operators on $\ell^2$.
\end{xpl}

We also have the following observations about the $\mathsf{b}$-norm of $0$-restrictions.  While the $\infty$-norm and $2$-norm of any $0$-restriction is bounded by the original $\infty$-norm or $2$-norm, with the $\mathsf{b}$-norm this only applies when restricting via the unit space.

\begin{prp}
For any $a\in\mathcal{S}^\rho$ and $X\subseteq\Gamma^0$,
\begin{equation}\label{bRestriction}
\|a_{\Gamma X}\|_\mathsf{b},\|a_{X\Gamma}\|_\mathsf{b}\leq\|a\|_\mathsf{b}.
\end{equation}
\end{prp}

\begin{proof}
For any $f\in\mathcal{F}^\rho$, note that
\[\|a_{\Gamma X}f\|_2=\|af_{X\Gamma}\|_2\leq\|a\|_\mathsf{b}\|f_{X\Gamma}\|_2\leq\|a\|_\mathsf{b}\|f\|_2.\]
This shows that $\|a_{\Gamma X}\|_\mathsf{b}\leq\|a\|_\mathsf{b}$ and then \eqref{infty2b} yields
\[\|a_{X\Gamma}\|_\mathsf{b}=\|(a_{X\Gamma})^*\|_\mathsf{b}=\|(a^*)_{\Gamma X}\|_\mathsf{b}\leq\|a^*\|_\mathsf{b}=\|a\|_\mathsf{b}.\qedhere\]
\end{proof}

\begin{xpl}
To see that the $\mathsf{b}$-norm of a $0$-restriction can be bigger than the original $\mathsf{b}$-norm, consider the trivial $\mathbb{C}$-bundle over the discrete principal groupoid with 4 elements and $2$ units.  Then elements of $\mathcal{S}^\rho$ are just $2$ by $2$ complex matrices and the $\mathsf{b}$-norm is just the usual operator norm.  Elementary calculations show that
\[\biggl\|\begin{bmatrix}1&1\\1&-1\end{bmatrix}\biggr\|\ =\ \sqrt{2}\ <\ \frac{1+\sqrt{5}}{2}\ =\ \biggl\|\begin{bmatrix}1&1\\1&0\end{bmatrix}\biggr\|.\]\\
\end{xpl}

For any $a,b\in\mathcal{S}^\rho$ with $\mathrm{supp}(a),\mathrm{supp}(b)\subseteq\Gamma^0$, let us write $a\leq b$ to mean that $a(x)\leq b(x)$ in the C*-algebra fibre $B_x$, for all $x\in\Gamma^0$.  When $\Gamma$ is a category, we have characteristic functions $1_X\leq1_{\Gamma^0}$, for any $X\subseteq\Gamma^0$.  Moreover, $a_{\Gamma X}=a1_X$ and so the above result can then be seen as a special case of the following.

\begin{prp}\label{StarSquareInequality}
For any $a,b,c\in\mathcal{S}^\rho$ such that $\mathrm{supp}(b)$ and $\mathrm{supp}(c)$ are slices,
\[bb^*\leq cc^*\qquad\Rightarrow\qquad\|ab\|_\mathsf{b}\leq\|ac\|_\mathsf{b}.\]
\end{prp}

\begin{proof}
For each $n\in\mathbb{N}$, let $f_n$ be the function on $\mathbb{R}_+$ such that $f_n(x)=n^2x$ for $0\leq x\leq 1/n$ and $f_n(x)=1/x$ for $x\geq1/n$.  Take any $\beta\in\mathrm{supp}(b)$ and $\gamma\in\mathrm{supp}(c)$ with $\mathsf{r}(\beta)=\mathsf{r}(\gamma)$ and let $b_\beta=b(\beta)$, $c_\gamma=c(\gamma)$ and $d_n=f_n(c_\gamma^*c_\gamma)c_\gamma^*b_\beta$.

First we claim that $c_\gamma d_n\rightarrow b_\beta$.  To see this note that, as $b_\beta b_\beta^*\leq c_\gamma c_\gamma^*$,
\begin{align*}
\|c_\gamma d_n-b_\beta\|^2&=\|(1-c_\gamma f_n(c_\gamma^*c_\gamma)c_\gamma^*)b_\beta b_\beta^*(1-c_\gamma f_n(c_\gamma^*c_\gamma)c_\gamma^*)\|\\
&\leq\|(1-c_\gamma f_n(c_\gamma^*c_\gamma)c_\gamma^*)c_\gamma c_\gamma^*(1-c_\gamma f_n(c_\gamma^*c_\gamma)c_\gamma^*)\|\\
&=\|(1-f_n(c_\gamma c_\gamma^*)c_\gamma c_\gamma^*)^2c_\gamma c_\gamma^*\|\\
&=\|g_n(c_\gamma c_\gamma^*)\|\\
&\rightarrow0,
\end{align*}
where $g_n$ is the function on $\mathbb{R}_+$ defined by $g_n(x)=(1-n^2x^2)^2x$, for $0\leq x\leq 1/n$ and $g_n(x)=0$ for $x\geq1/n$, thus proving the claim.

We next claim that $\|c_\gamma d_n\|\leq1$, for all $n\in\mathbb{N}$.  To see this, note that
\[\|c_\gamma d_n\|^2=\|f_n(c_\gamma^*c_\gamma)c_\gamma^*b_\beta b_\beta^*c_\gamma^*f_n(c_\gamma^*c_\gamma)\|\leq\|f_n(c_\gamma^*c_\gamma)c_\gamma^*c_\gamma c_\gamma^*c_\gamma^*f_n(c_\gamma^*c_\gamma)\|=\|h_n(c_\gamma^*c_\gamma)\|\,\]
where $h_n$ is the function on $\mathbb{R}_+$ defined by $h_n(x)=n^4x^4$, for $0\leq x\leq 1/n$ and $g_n(x)=1$ for $x\geq1/n$.  As $h_n(x)\leq1$, for all $x$, this proves the claim.

Now, for any $r<\|ab\|_\mathsf{b}$, we have some $f\in\mathcal{F}^\rho$ with $\|f\|_2\leq1$ and $\|abf\|_2>r$.  For each $x\in\mathsf{r}[\mathrm{supp}(bf)]$, we have unique $\beta\in\mathrm{supp}(\beta)$ and $\gamma\in\mathrm{supp}(\gamma)$ with $x=\mathsf{r}(\beta)=\mathsf{r}(\gamma)$.  As above, we then have $(d_n)\subseteq B_{\gamma^{-1}\beta}$ with $\|d_n\|\leq1$, for all $n\in\mathbb{N}$, and $c(\gamma)d_n\rightarrow b(\beta)$.  As $\mathrm{supp}(b)$ is a slice and hence $\mathsf{r}[\mathrm{supp}(bf)]$ is finite, this means we have $d\in\mathcal{F}^\rho$ such that $\mathrm{supp}(d)$ is a slice, $\|d\|_\infty\leq1$ and
\[r<\|acdf\|_2\leq\|ac\|_\mathsf{b}\|df\|_2\leq\|ac\|_\mathsf{b}\|d\|_\infty\|f\|_2\leq\|ac\|_\mathsf{b},\]
thanks to \eqref{Infinity2}.  As $r<\|ab\|_\mathsf{b}$ was arbitrary, this shows that $\|ab\|_\mathsf{b}\leq\|ac\|_\mathsf{b}$.
\end{proof}

\subsection{The $\mathsf{i}$-Norm}

Lastly, we introduce a variant of Renault's $\mathsf{i}$-norm, defined to be the maximum of the $1$-norm given in \eqref{1Norm} and its dual, i.e. for all $a\in\mathcal{S}^\rho$,
\[\|a\|_\mathsf{i}=\|a\|_1\vee\|a^*\|_1.\]
The $\mathsf{i}$-norm will not be needed for any of the general theory we continue to develop in the next section -- its significance stems from the fact that it bounds the $\mathsf{b}$-norm and is often easier to compute.  Showing that the $\mathsf{i}$-norm is finite can thus provide a convenient way of verifying that a particular section lies in $\mathcal{S}^\rho_\mathsf{b}$.

First we need a lemma showing that we have a kind of square root function even on non-unit fibres of our bundle.  First note that, as each unit fibre is a C*-algebra, for any $b\in B$, we may define $|b|\in B_{\mathsf{s}(\rho(b))}$ in the standard way by
\[|b|=\sqrt{b^*b}.\]
Also note that we can `shift' applications of the continuous functional calculus in unit fibres as usual, i.e. if $f$ is any continuous function $[0,\infty)$ with $f(0)=0$ then
\begin{equation}\label{Shift}
f(bb^*)b=bf(b^*b).
\end{equation}
Indeed, \eqref{Shift} is certainly true if $f$ is a polynomial with no constant term and hence extends to arbitrary continuous $f$ with $f(0)=0$ via limits.

\begin{lem}
For every $b\in B$, we have $\underline{b}\in B_{\rho(b)}$ with
\[|b|=\underline{b}^*\underline{b},\qquad|b^*|=\underline{b}\,\underline{b}^*\qquad\text{and}\qquad b=\underline{b}\sqrt{|b|}=\sqrt{|b^*|}\underline{b}.\]
\end{lem}

\begin{proof}
Let $(f_n)$ be an increasing sequence of continuous functions on $[0,\infty]$ such that $f_n(x)\rightarrow1/\sqrt{x}$, for all $x>0$.  For all $n\in\mathbb{N}$, let $c_n=f_n(|b|)\in B_{\mathsf{s}(\rho(b))}$ so $|b|c_n=c_n|b|\rightarrow\sqrt{|b|}$.  For all $m,n\in\mathbb{N}$, it follows that
\begin{align*}
\|(bc_m-bc_n)\|^2&=\|(bc_m-bc_n)^*(bc_m-bc_n)\|\\
&\leq\||b|c_m\|\||b|(c_m-c_n)\|+\||b|c_n\|\||b|(c_m-c_n)\|\\
&\rightarrow0,
\end{align*}
as $m,n\rightarrow0$, i.e. $(bc_n)$ is Cauchy and hence has limit $\underline{b}\in B_b$.  Then
\begin{align*}
\|b-bc_n\sqrt{|b|}\|^2&=\|(b-bc_n\sqrt{|b|})^*(b-bc_n\sqrt{|b|})\|\\
&\leq\||b|^2-2|b|^\frac{5}{2}c_n+|b|^3c_n^2\|\\
&\rightarrow0,
\end{align*}
showing that $\underline{b}\sqrt{|b|}=\lim_nbc_n\sqrt{|b|}=b$.  Likewise,
\begin{align*}
\|b-\sqrt{|b^*|}bc_n\|^2&=\|(b-\sqrt{|b^*|}bc_n)^*(b-\sqrt{|b^*|}bc_n)\|\\
&\leq\||b|^2-b^*\sqrt{|b^*|}bc_n-c_nb^*\sqrt{|b^*|}b+c_nb^*|b^*|bc_n\|\\
&=\||b|^2-b^*(bb^*)^{\frac{1}{4}}bc_n-c_nb^*(bb^*)^{\frac{1}{4}}b+c_nb^*\sqrt{bb^*}bc_n\|\\
&=\||b|^2-(b^*b)^{\frac{1}{4}}b^*bc_n-c_n(b^*b)^{\frac{1}{4}}b^*b+c_n\sqrt{b^*b}b^*bc_n\|\\
&=\||b|^2-2|b|^\frac{5}{2}c_n+|b|^3c_n^2\|\\
&\rightarrow0,
\end{align*}
so $\sqrt{|b^*|}\underline{b}=\lim_n\sqrt{|b^*|}bc_n=b$.  Similarly, $\||b|-(bc_n)^*bc_n\|=\||b|-|b|^2c_n^2\|\rightarrow0$ so $\underline{b}^*\underline{b}=\lim_n(bc_n)^*bc_n=|b|$ and hence $|b^*|=\sqrt{bb^*}=\sqrt{\underline{b}|b|\underline{b}^*}=\sqrt{\underline{b}\,\underline{b}^*\underline{b}\,\underline{b}^*}=\underline{b}\,\underline{b}^*$.
\end{proof}

Now we can show that the $\mathsf{i}$-norm does indeed dominate the $\mathsf{b}$-norm.

\begin{prp}
For all $a\in\mathcal{S}^\rho$,
\[\|a\|_\mathsf{b}\leq\|a\|_\mathsf{i}.\]
\end{prp}

\begin{proof}
For any $f,g\in\mathcal{F}^\rho$, and $x\in\Gamma^0$, we can apply the Cauchy-Schwarz inequality to the right Hilbert $B_x$-module $\bigoplus_{(\alpha,\beta)\in x\Gamma\times\Gamma x}B_\beta$ to obtain
\begin{align*}
\|g^*af(x)\|^2&=\Big\|\sum_{\substack{\alpha\in x\Gamma\\\beta\in\Gamma x}}g^*(\alpha)a(\alpha^{-1}\beta^{-1})f(\beta)\Big\|^2\\
&=\Big\|\sum_{\substack{\alpha\in x\Gamma\\\beta\in\Gamma x}}g^*(\alpha)\underline{a(\alpha^{-1}\beta^{-1})}\sqrt{|a(\alpha^{-1}\beta^{-1})|}f(\beta)\Big\|^2\\
&\leq\Big\|\sum_{\substack{\alpha\in x\Gamma\\\beta\in\Gamma x}}g^*(\alpha)|a(\alpha^{-1}\beta^{-1})^*|g^*(\alpha)^*\Big\|\Big\|\sum_{\substack{\alpha\in x\Gamma\\\beta\in\Gamma x}}f(\beta)^*|a(\alpha^{-1}\beta^{-1})|f(\beta)\Big\|\\
&\leq\Big\|\sum_{\substack{\alpha\in x\Gamma\\\beta\in\Gamma x}}\|a(\alpha^{-1}\beta^{-1})\|g^*(\alpha)g^*(\alpha)^*\Big\|\Big\|\sum_{\substack{\alpha\in x\Gamma\\\beta\in\Gamma x}}\|a(\alpha^{-1}\beta^{-1})\|f(\beta)^*f(\beta)\Big\|\\
&\leq\Big\|\sum_{\alpha\in x\Gamma}\|a^*\|_1g^*(\alpha)g^*(\alpha)^*\Big\|\Big\|\sum_{\beta\in\Gamma x}\|a\|_1f(\beta)^*f(\beta)\Big\|\\
&=\|a^*\|_1\|g^*g(x)\|\|a\|_1\|f^*f(x)\|\\
&\leq\|a^*\|_\mathsf{i}^2\|g\|_2^2\|f\|_2^2.
\end{align*}
Thus $\|\Phi(g^*af)\|_\infty\leq\|a^*\|_\mathsf{i}\|g\|_2\|f\|_2$ and hence $\|a\|_\mathsf{b}\leq\|a\|_\mathsf{i}$.
\end{proof}

\section{Algebras and Modules}

Again we make the following standing assumption.
\begin{center}
\textbf{Throughout the rest of this section $\rho:B\twoheadrightarrow\Gamma$ is a Fell bundle.}
\end{center}

\subsection{The Big Hilbert Module}\label{TheBigHilbertModule}

Let us define
\begin{align*}
\mathcal{H}^\rho&=\{a\in\mathcal{S}_2^\rho:a^*a\text{ is defined}\}.\\
\mathcal{D}^\rho&=\{a\in\mathcal{S}_\infty^\rho:\mathrm{supp}(a)\subseteq\Gamma^0\}\\
&=\{\Phi(a):a\in\mathcal{S}_\infty^\rho\}.
\end{align*}
Identifying $0$-restrictions with restrictions, $\mathcal{D}^\rho$ just consists of all bounded sections of the C*-bundle defined by restricting $\rho$ to the preimage of the unit space $\rho^{-1}[\Gamma^0]$.  As convolution in $\mathcal{D}^\rho$ reduces to the pointwise product, $\mathcal{D}^\rho$ is a C*-algebra.

\begin{thm}\label{Hilbert}
$\mathcal{H}^\rho$ is a right Hilbert module over $\mathcal{D}^\rho$ with inner product
\[\langle a,b\rangle=\Phi(a^*b).\]
\end{thm}

\begin{proof}
By \autoref{a*b}, the given inner product is well-defined with values in $\mathcal{D}^\rho$.  Moreover, for any $a,b\in\mathcal{H}^\rho$, $(a+b)^*(a+b)=a^*a+a^*b+b^*a+b^*b$ is defined, again by \autoref{a*b}, and hence $a+b\in\mathcal{H}^\rho$.  For any $d\in\mathcal{D}^\rho$, we see that $bd$ and $(bd)^*bd=d^*b^*bd$ are also defined and $\|bd\|_2\leq\|b\|_2\|d\|_\infty=\|b\|_2\|d\|_2$, which means that $bd\in\mathcal{H}^\rho$.  Moreover, as $\mathrm{supp}(d)\subseteq\Gamma^0$, we also see that
\[\langle a,bd\rangle=\Phi(a^*bd)=\Phi(a^*b)d=\langle a,b\rangle d.\]
The other properties of a right inner product $\mathcal{D}^\rho$-module are immediate.  It only remains to show that $\mathcal{H}^\rho$ is complete with respect to the inner product norm.

By \autoref{a*aDefined}, the inner product norm is the same as the $2$-norm, i.e.
\[\|a\|_2=\sqrt{\|\langle a,a\rangle\|}.\]
Thus it suffices to show that $\mathcal{H}^\rho$ is closed in the Banach space $\mathcal{S}_2^\rho$.  To see this, take $a\in\mathcal{S}_2^\rho$ in the closure of $\mathcal{H}^\rho$, so we have $(a_n)\subseteq\mathcal{H}^\rho$ with $\|a-a_n\|_2\rightarrow0$.  For any $x\in\Gamma^0$, $F\subset\Gamma x$ and $n\in\mathbb{N}$,
\[\|a_{\Gamma x\setminus F}\|_2\leq\|(a-a_n)_{\Gamma x\setminus F}\|_2+\|(a_n)_{\Gamma x\setminus F}\|_2\leq\|a-a_n\|_2+\|(a_n)_{\Gamma x\setminus F}\|_2.\]
Thus $\lim_{F\subset\Gamma x}\|a_{\Gamma x\setminus F}\|_2\leq\|a-a_n\|_2\rightarrow0$, for all $x\in\Gamma^0$, so $a^*a$ is defined, by \autoref{a*aDefined}.  This means $a\in\mathcal{H}^\rho$, showing that $\mathcal{H}^\rho$ is closed, as required.
\end{proof}

Convolution then turns any $a\in\mathcal{S}_\mathsf{b}^\rho$ into a bounded linear map from $\mathcal{H}^\rho$ to $\mathcal{S}_2^\rho$.

\begin{prp}\label{abinH}
If $a\in\mathcal{S}_\mathsf{b}^\rho$ and $b\in\mathcal{H}^\rho$ then $ab\in\mathcal{S}_2^\rho$, specifically
\begin{equation}\label{b2}
\|ab\|_2\leq\|a\|_\mathsf{b}\|b\|_2.
\end{equation}
If $a^*a$ is also defined, i.e. if $a\in\mathcal{H}^\rho$ too, then $ab\in\mathcal{H}^\rho$ as well.
\end{prp}

\begin{proof}
By \autoref{a*b}, $ab$ is defined, as $\|a^*\|_2\leq\|a^*\|_\mathsf{b}=\|a\|_\mathsf{b}<\infty$.  Moreover, for all $x\in\Gamma^0$ and $F\subseteq G\subset\Gamma x$,
\[\|ab_G-ab_F\|_2=\|ab_{G\setminus F}\|_2\leq\|a\|_\mathsf{b}\|b_{G\setminus F}\|_2\leq\|a\|_\mathsf{b}\|b_{\Gamma x\setminus F}\|_2.\]
This shows $(ab_F)_{F\subset\Gamma x}$ is $2$-Cauchy, as $\lim_{F\subset\Gamma x}\|b_{\Gamma x\setminus F}\|_2=0$.  By \autoref{bBanach}, $(ab_F)_{F\subset\Gamma x}$ has a limit in $\mathcal{S}_2^\rho$, which must also be the pointwise limit $ab_{\Gamma x}$, i.e.
\begin{equation}\label{2limitF}
\lim_{F\subset\Gamma x}\|ab_{\Gamma x}-ab_F\|_2=0.
\end{equation}
Also $\|ab_{\Gamma x}\|_2=\lim_{F\subset\Gamma x}\|ab_F\|_2\leq\sup_{F\subset\Gamma x}\|a\|_\mathsf{b}\|b_F\|_2\leq\|a\|_\mathsf{b}\|b\|_2$ so
\[\|ab\|_2=\sup_{x\in\Gamma^0}\|ab_{\Gamma x}\|_2\leq\|a\|_\mathsf{b}\|b\|_2.\]
In particular, $\|ab\|_2<\infty$ so $ab\in\mathcal{S}_2^\rho$.

If $a\in\mathcal{H}^\rho$ too then $ab_F\in\mathcal{H}^\rho$, for all $F\subset\Gamma x$.  As $ab_{\Gamma x}$ is the $2$-limit of $(ab_F)_{F\subset\Gamma x}$ and $\mathcal{H}^\rho$ is $2$-complete, by \autoref{Hilbert}, it follows that $ab_{\Gamma x}\in\mathcal{H}^\rho$, for all $x\in\Gamma^0$.  Thus $ab\in\mathcal{H}^\rho$, by \autoref{a*aDefined}.
\end{proof}

We also have a corresponding associativity result.

\begin{prp}\label{AssociativeProduct}
If $a^*,c\in\mathcal{H}^\rho$ and $b\in\mathcal{S}_\mathsf{b}^\rho$ then $(ab)c=a(bc)$.
\end{prp}

\begin{proof}
Note $bc,b^*a^*\in\mathcal{S}_2^\rho$, by \autoref{abinH}, so $(ab)c$ and $a(bc)$ are defined, by \autoref{a*b}.  Moreover, for any $x\in\Gamma^0$ and $F\subset\Gamma x$, \eqref{infty22} and \eqref{b2} yield
\begin{align*}
\|(ab)c_{\Gamma x}-(ab)c_F\|_\infty&=\|(ab)c_{\Gamma x\setminus F}\|_\infty\leq\|b^*a^*\|_2\|c_{\Gamma x\setminus F}\|_2\leq\|b^*\|_\mathsf{b}\|a^*\|_2\|c_{\Gamma x\setminus F}\|_2\\
\|a(bc_{\Gamma x})-a(bc_F)\|_\infty&=\|a(bc_{\Gamma x\setminus F})\|_\infty\leq\|a^*\|_2\|bc_{\Gamma x\setminus F}\|_2\leq\|a^*\|_2\|b\|_\mathsf{b}\|c_{\Gamma x\setminus F}\|_2.
\end{align*}
By \eqref{abcF}, $(ab)c_F=a(bc_F)$ so $\|(ab)c_{\Gamma x}-a(bc_{\Gamma x})\|_\infty\leq2\|a^*\|_2\|b\|_\mathsf{b}\|c_{\Gamma x\setminus F}\|_2$.  Then \eqref{a*a} yields $((ab)c)_{\Gamma x}=(ab)c_{\Gamma x}=a(bc_{\Gamma x})=(a(bc))_{\Gamma x}$.  As $x$ was arbitrary, it follows that $(ab)c=a(bc)$ everywhere on $\Gamma$.
\end{proof}

\subsection{The Big C*-Algebra}

Next we define
\[\mathcal{B}^\rho=\mathcal{S}_\mathsf{b}^\rho\cap\mathcal{H}^\rho\cap\mathcal{H}^{\rho*}=\{a\in\mathcal{S}_\mathsf{b}^\rho:a^*a\text{ and }aa^*\text{ are defined}\}.\]

\begin{thm}\label{BigCstar}
$\mathcal{B}^\rho$ is a C*-algebra with norm $\|\cdot\|_\mathsf{b}$.
\end{thm}

\begin{proof}
By \eqref{||ab||_b} and \autoref{abinH}, $\mathcal{B}^\rho$ is closed under products.  By \autoref{AssociativeProduct}, products are also associative on $\mathcal{B}^\rho$.  We also immediately see that $(ab)^*=b^*a^*$ and $a(b+c)=ab+ac$, for all $a,b,c\in\mathcal{B}^\rho$.  The $\mathsf{b}$-norm is submultiplicative, by \eqref{||ab||_b}, and satisfies the C*-norm condition, by \autoref{a*aDefined}.  As $\mathcal{S}_\mathsf{b}^\rho$ and $\mathcal{H}^\rho$ are Banach spaces, and the $\mathsf{b}$-norm dominates the $2$-norm, it follows that $\mathcal{B}^\rho$ is also a Banach space and hence a C*-algebra.
\end{proof}

We call $\mathcal{B}^\rho$ the \emph{big C*-algebra} of $\rho$ (not to be confused with the maximal/full C*-algebra, which is something else).  Indeed, $\mathcal{B}^\rho$ is immediately seen to be the biggest C*-algebra consisting of sections of $\rho$ under convolution and the $\mathsf{b}$-norm.

\begin{prp}\label{AssociativeProduct2}
If $a\in\mathcal{S}_\mathsf{b}^\rho$, $b\in\mathcal{B}^\rho$ and $c\in\mathcal{H}^\rho$ then $(ab)c=a(bc)$.
\end{prp}

\begin{proof}
This is proved exactly as in \autoref{AssociativeProduct}, once we note that $(ab)^*=b^*a^*$ is a well-defined element of $\mathcal{S}_\mathsf{b}^\rho\subseteq\mathcal{S}_2^\rho$, by \autoref{a*b} and \eqref{||ab||_b}.
\end{proof}

Say $M$ is a bimodule over an algebra $A$.  We say $M$ is a \emph{Banach bimodule} over $A$ if both $M$ and $A$ are also Banach spaces with $\|ab\|\leq\|a\|\|b\|$, whenever $a$ and/or $b$ are in $A$.  We say $M$ is a \emph{Banach *-bimodule} over $A$ if we also have involutions on both $M$ and $A$ such that $\|a^*\|=\|a\|$ and $(ab)^*=b^*a^*$, whenever $a$ and/or $b$ are in $A$.  In fact, we are interested in certain Banach *-bimodules over subspaces, i.e. where $A\subseteq M$ with the involution and Banach space structure on $A$ induced by $M$.

\begin{cor}
$\mathcal{S}_\mathsf{b}^\rho$ is a Banach *-bimodule over $\mathcal{B}^\rho$.
\end{cor}

\begin{proof}
By \autoref{bBanach}, $\mathcal{S}_\mathsf{b}^\rho$ is a Banach space.  By \eqref{||ab||_b} and \autoref{a*b}, whenever $a\in\mathcal{S}_\mathsf{b}^\rho$ and $b\in\mathcal{B}^\rho$ or vice versa, $ab$ is a well-defined element of $\mathcal{S}_\mathsf{b}^\rho$ with $(ab)^*=b^*a^*$ and $\|ab\|_\mathsf{b}\leq\|a\|_\mathsf{b}\|b\|_\mathsf{b}$.  Also $a(bc)=(ab)c$ when at least two of these elements are in $\mathcal{B}^\rho$ and the other is in $\mathcal{S}_\mathsf{b}^\rho$, thanks to \autoref{AssociativeProduct} and \autoref{AssociativeProduct2}.  Distributivity of products over sums is immediate.
\end{proof}

For any Hilbert module $H$, we denote the adjointable operators on $H$ by $\mathsf{B}(H)$.  By \cite[Proposition 2.21]{RaeburnWilliams1998}, $\mathsf{B}(H)$ is a C*-algebra w.r.t. the operator norm.

\begin{thm}\label{LeftRegRep}
Each $a\in\mathcal{S}_\mathsf{b}^\rho$ defines a bounded operator $a_\mathsf{B}:\mathcal{H}^\rho\rightarrow\mathcal{S}_2^\rho$ by
\[a_\mathsf{B}(b)=ab.\]
Moreover, $a\mapsto a_\mathsf{B}$ restricts to an isomorphism on $\mathcal{B}^\rho$ to a C*-subalgebra of $\mathsf{B}(\mathcal{H}^\rho)$.
\end{thm}

\begin{proof}
If $a\in\mathcal{S}_\mathsf{b}^\rho$ then, by \autoref{abinH}, $a_\mathsf{B}$ is indeed a well-defined $2$-bounded operator from $\mathcal{H}^\rho$ to $\mathcal{S}_2^\rho$ with $\|a_\mathsf{B}\|\leq\|a\|_\mathsf{b}$.  If $a\in\mathcal{B}^\rho$ then $\mathrm{ran}(a_\mathsf{B})\subseteq\mathcal{H}^\rho$, again by \autoref{abinH}.  For any $b,c\in\mathcal{H}^\rho$, \autoref{AssociativeProduct} then yields
\[\langle ab,c\rangle=(ab)^*c=(b^*a^*)c=b^*(a^*c)=\langle b,a^*c\rangle.\]
So if $a\in\mathcal{B}^\rho$ then $a_\mathsf{B}$ is adjointable, specifically $(a_\mathsf{B})^*=(a^*)_\mathsf{B}$, i.e. $a_\mathsf{B}\in\mathsf{B}(\mathcal{H}^\rho)$.

Next note that $\|a\|_\mathsf{b}\leq\|a_\mathsf{B}\|$, as $\mathcal{F}^\rho\subseteq\mathcal{H}^\rho$, so in fact $\|a_\mathsf{B}\|=\|a\|_\mathsf{b}$, i.e. $a\mapsto a_\mathsf{B}$ is an isometry.  Certainly $(a+b)_\mathsf{B}=a_\mathsf{B}+b_\mathsf{B}$, for any $a,b\in\mathcal{S}_\mathsf{b}^\rho$, and also $(ab)_\mathsf{B}=a_\mathsf{B}b_\mathsf{B}$, as long as $a\in\mathcal{B}^\rho$ or $b\in\mathcal{B}^\rho$, by \autoref{AssociativeProduct} and \autoref{AssociativeProduct2} respectively.  In particular, $a\mapsto a_\mathsf{B}$ is an isomorphism from $\mathcal{B}^\rho$ onto a C*-subalgebra of $\mathsf{B}(\mathcal{H}^\rho)$.
\end{proof}

By using the above representation, we can show that the $\mathsf{b}$-norm of any $c\in\mathcal{B}^\rho$ can be calculated just from $f\in\mathcal{F}^\rho$ taking values generated by the range of $c$.  This will be needed later in the proof of \autoref{betarho} below.

\begin{lem}
If $C\subseteq B$ be a *-subsemigroupoid of $B$ such that $C_\gamma=C\cap B_\gamma$ is a subspace of $B_\gamma$, for all $\gamma\in\Gamma$, then, for any $c\in\mathcal{B}^\rho$ with $\mathrm{ran}(c)\subseteq C$,
\begin{equation}\label{bC}
\|c\|_\mathsf{b}=\sup\{\|cf\|_2:f\in\mathcal{F}^\rho,\|f\|_2=1\text{ and }\mathrm{ran}(f)\subseteq C\}.
\end{equation}
\end{lem}

\begin{proof}
Let $A=C^*(c)$ be the C*-subalgebra of $\mathcal{B}^\rho$ generated by $c$.  Further let
\[H=\mathrm{cl}_2\{h\in\mathcal{H}^\rho:\mathrm{ran}(h)\subseteq C\}.\]
Note that $H$ is a Hilbert $D$-module, where $D=\mathcal{D}^\rho\cap H$, which is also invariant under multiplication by $c$ and $c^*$.  Thus $a\mapsto a_\mathsf{B}|_H$ is a representation of $A$ as a C*-subalgebra of $\mathsf{B}(H)$.  To see that this representation is faithful, take any $a\in A\setminus\{0\}$ so we have $\alpha\in\Gamma$ with $a(\alpha)\neq0$.  Defining $h\in H$ by $h(\alpha^{-1})=a(\alpha)^*$ and $h(\gamma)=0$, for all $\gamma\neq\alpha^{-1}$, we see that $ah(\mathsf{r}(\alpha))=a(\alpha)h(\alpha^{-1})=a(\alpha)a(\alpha)^*\neq0$ so, in particular, $a_\mathsf{B}|_H\neq0$, as required.  As faithful representations are isometric,
\[\|c\|_\mathsf{b}=\|c_\mathsf{B}|_H\|=\sup\{\|cf\|_2:f\in\mathcal{F}^\rho,\|f\|_2=1\text{ and }\mathrm{ran}(f)\subseteq C\}.\qedhere\]
\end{proof}

\subsection{The Reduced C*-Algebra}

Recall the definition of $\mathcal{C}_\mathsf{c}^\rho$ from \eqref{CcDef}.

\begin{thm}
We have a C*-subalgebra of the big C*-algebra $\mathcal{B}^\rho$ given by
\begin{equation}\label{ReducedDefinition}
\mathcal{C}^\rho_\mathsf{r}=\mathrm{cl}_\mathsf{b}(\mathcal{C}^\rho_\mathsf{c}).
\end{equation}
\end{thm}

\begin{proof}
Any bicompact $K$ is a slice so, for any $a\in\mathcal{C}^\rho(K)$, certainly $a^*a$ and $aa^*$ are defined.  Also $\|a\|_\mathsf{b}=\|a\|_2=\|a\|_\infty<\infty$, by \autoref{C0} \eqref{C01}, \autoref{21} and \autoref{infty2bprp}, so $\mathcal{C}^\rho_\mathsf{c}\subseteq\mathcal{B}^\rho$.  As bicompact subsets are closed under taking inverses, $\mathcal{C}^\rho_\mathsf{c}$ and hence $\mathcal{C}^\rho_\mathsf{r}$ is a *-invariant subspace of $\mathcal{B}^\rho$.  It only remains to show that $\mathcal{C}^\rho_\mathsf{c}$ and hence $\mathcal{C}^\rho_\mathsf{r}$ is closed under products.

First we claim that, for any $\Delta\subseteq\Gamma$ and bicompact $K\subseteq\Gamma$,
\[\mathcal{C}^\rho(\Delta)\mathcal{C}^\rho(K)\subseteq\mathcal{C}^\rho(\Delta K)\]
To see this, take any $a\in\mathcal{C}^\rho(\Delta)$ and $b\in\mathcal{C}^\rho(K)$.  For any $\gamma\in\Delta K$, we have unique $\alpha_\gamma\in\Delta$ and $\beta_\gamma\in K$ with $\gamma=\alpha_\gamma\beta_\gamma$ and hence $ab(\gamma)=a(\alpha_\gamma)b(\beta_\gamma)$.  For any open slice $O$ containing $K$, we see that $\beta_\gamma=\mathsf{s}|_O^{-1}(\mathsf{s}(\gamma))$ and $\alpha_\gamma=\gamma\beta_\gamma^{-1}$, which are both seen to be continuous functions in $\gamma$ on the open set $\Gamma O$ containing $\Delta K$.  As $\Delta\subseteq\mathcal{C}(a)$ and $K\subseteq\mathcal{C}(b)$, it follows that $\Delta K\subseteq\mathcal{C}(ab)$, proving the claim.  In particular, this holds when $\Delta$ is also bicompact, in which case $\Delta K$ is bicompact too, by \autoref{CompactProducts}.  Thus $\mathcal{C}^\rho_\mathsf{c}\mathcal{C}^\rho_\mathsf{c}\subseteq\mathcal{C}^\rho_\mathsf{c}$ and hence $\mathcal{C}^\rho_\mathsf{r}\mathcal{C}^\rho_\mathsf{r}\subseteq\mathcal{C}^\rho_\mathsf{r}$.
\end{proof}

We call $\mathcal{C}_\mathsf{r}^\rho$ the \emph{reduced C*-algebra} of $\rho$.  It is indeed isomorphic to the more classical version of the reduced C*-algebra $\mathcal{C}_\mathsf{r}^*(\Gamma)$ obtained from the left regular representation of $\mathcal{C}^\rho_\mathsf{c}$ in \autoref{LeftRegRep} above, a fact we record here for future reference.

\begin{cor}
The reduced C*-algebra $\mathcal{C}_\mathsf{r}^\rho$ defined above is isomorphic to the classical reduced C*-algebra $\mathcal{C}_\mathsf{r}^*(\Gamma)=\mathrm{cl}(\{a_\mathsf{B}:a\in\mathcal{C}^\rho_\mathsf{c}\})\subseteq\mathsf{B}(\mathcal{H}^\rho)$.
\end{cor}

\begin{proof}
By \autoref{LeftRegRep}, $a\mapsto a_\mathsf{B}$ is an isomorphism on $\mathcal{B}^\rho$ so this is immediate.
\end{proof}

In contrast to $\mathcal{C}_\mathsf{r}^*(\Gamma)$, the advantage of $\mathcal{C}_\mathsf{r}^\rho$ is that it is already a C*-algebra of sections under convolution from the outset -- there is no need to define any $j$-map from operators back to concrete sections and then verify that $j$ turns composition of operators into convolution of sections.

Another advantage is that we are free to consider other algebras and modules of concrete sections under convolution.

\begin{thm}
We have Banach a *-bimodule over $\mathcal{C}^\rho_\mathsf{r}$ given by
\[\mathcal{C}^\rho_\mathsf{b}=\mathcal{C}^\rho\cap\mathcal{S}^\rho_\mathsf{b}=\{a\in\mathcal{S}^\rho:a\text{ is continuous everywhere and }\|a\|_\mathsf{b}<\infty\}.\]
\end{thm}

\begin{proof}
First we argue as in the proof of \cite[Theorem 2.10]{Bice2023} to show that a product of $a\in\mathcal{C}^\rho$ and $s\in\mathcal{C}^\rho(K)$ is again continuous when $K$ is bicompact.  To see this, take an open slice $O$ containing $K$ and note that $s$ is then continuous on $O$ -- by the definition of $\mathcal{C}^\rho(K)$, $s$ is continuous on $K$ and zero on $O\setminus K$ and hence continuous on $O\setminus K$, as $K$ is biclosed and hence closed in $O$, implying $O\setminus K$ is open.  The map $\gamma\mapsto\gamma_O=\mathsf{r}|_O^{-1}(\mathsf{r}(\gamma))$ is also continuous from $O\Gamma=\mathsf{r}^{-1}[\mathsf{r}[O]]$ to $O$.  Thus $sa$ is also continuous on $O\Gamma$, as $sa(\gamma)=s(\gamma_O)a(\gamma_O^{-1}\gamma)$, for all $\gamma\in O\Gamma$.  On the other hand, $K\Gamma=\mathsf{r}^{-1}[\mathsf{r}[K]]$ is closed subset of $O\Gamma$, as $K$ is biclosed and hence $\mathsf{r}[K]$ is closed.  As $sa$ takes zero values on the open set $\Gamma\setminus K\Gamma$, it is also continuous there, i.e. $sa\in\mathcal{C}^\rho$.  By a dual argument, $as\in\mathcal{C}^\rho$, which proves the claim.  As $\mathcal{C}^\rho$ is closed under linear combinations, this extends to $s\in\mathcal{C}^\rho_\mathsf{c}$.  By \eqref{||ab||_b} and the fact that $\mathcal{C}^\rho$ is also $\infty$-closed and hence $\mathsf{b}$-closed, this extends to $s\in\mathcal{C}^\rho_\mathsf{r}$, i.e. $\mathcal{C}_\mathsf{r}^\rho\mathcal{C}^\rho\cup\mathcal{C}^\rho\mathcal{C}_\mathsf{r}^\rho\subseteq\mathcal{C}^\rho$.  Again by \eqref{||ab||_b}, we can then replace $\mathcal{C}^\rho$ by $\mathcal{C}^\rho_\mathsf{b}$, i.e.
\[\mathcal{C}_\mathsf{r}^\rho\mathcal{C}_\mathsf{b}^\rho\cup\mathcal{C}_\mathsf{b}^\rho\mathcal{C}_\mathsf{r}^\rho\subseteq\mathcal{C}_\mathsf{b}^\rho.\]
As $\mathcal{C}_\mathsf{r}^\rho$ and $\mathcal{C}_\mathsf{b}^\rho$ are *-invariant, it follows that $\mathcal{C}_\mathsf{b}^\rho$ is a Banach *-bimodule over $\mathcal{C}_\mathsf{r}^\rho$.
\end{proof}

Let $\rho_0:\rho^{-1}[\Gamma^0]\twoheadrightarrow\Gamma^0$ be the C*-bundle obtained by restricting $\rho$ to the units in the base groupoid, i.e. $\rho_0=\rho|_{\rho^{-1}[\Gamma^0]}$ from which we obtain the C*-subalgebra
\[\mathcal{C}^{\rho_0}_\mathsf{r}=\mathrm{cl}_\mathsf{b}(\mathcal{C}^{\rho_0}_\mathsf{c})=\mathrm{cl}_\mathsf{b}(\mathrm{span}\{a\in\mathcal{C}^\rho(K):K=\mathrm{cl}(K)\cap\Gamma^0\text{ is compact}\})\subseteq\mathcal{C}^\rho_0(\Gamma^0)\cap\mathcal{C}_\mathsf{r}^\rho.\]
While $\mathcal{C}^{\rho_0}_\mathsf{r}$ is smaller than $\mathcal{C}_\mathsf{r}^\rho$, it will at least contain an approximate unit for $\mathcal{C}_\mathsf{r}^\rho$, i.e. a net $(z_\lambda)$ in the positive unit ball with $z_\lambda a\rightarrow a$, for all $a\in\mathcal{C}_\mathsf{r}^\rho$.

\begin{prp}
Any approximate unit for $\mathcal{C}^{\rho_0}_\mathsf{r}$ is an approximate unit for $\mathcal{C}^\rho_\mathsf{r}$.
\end{prp}

\begin{proof}
Take an approximate unit $(z_\lambda)$ for $\mathcal{C}^{\rho_0}_\mathsf{r}$.  For any $a\in\mathcal{C}^\rho(K)$, where $K$ is bicompact, $aa^*\in\mathcal{C}^{\rho_0}_\mathsf{r}$ and hence, letting $1$ be the unit in the unitisation of $\mathcal{C}^\rho_\mathsf{r}$,
\[\|z_\lambda a-a\|^2=\|(1-z_\lambda)aa^*(1-z_\lambda)\|\leq\|(1-z_\lambda)aa^*\|\rightarrow0.\]
This then immediately extends to linear combinations and their $\mathsf{b}$-limits.  In other words, $z_\lambda a\rightarrow a$, for all $a\in\mathcal{C}_\mathsf{r}^\rho$, showing that $(z_\lambda)$ is an approximate unit for $\mathcal{C}_\mathsf{r}^\rho$.
\end{proof}

Let us also define the (potentially larger) \emph{reduced diagonal} C*-subalgebra
\[\mathcal{D}^\rho_\mathsf{r}=\{a\in\mathcal{C}^\rho_\mathsf{r}:\mathrm{supp}(a)\subseteq\Gamma^0\},\]

\subsection{The Multiplier Algebras}

We define the \emph{left $\rho$-multipliers} by
\[\mathcal{L}^\rho=\{a\in\mathcal{S}^\rho:\forall z\in\mathcal{D}^\rho_\mathsf{r}\ (az\in\mathcal{C}_\mathsf{r}^\rho)\}.\]
To say more about these, we need to assume $\rho$ is `left-full'.

\begin{dfn}
We call $\rho$ \emph{left-full} or \emph{right-full} respectively if, for every $b\in B$,
\begin{align*}
\tag{Left-Full}\inf\{\|d(\mathsf{r}(\rho(b)))b-b\|:d\in\mathcal{D}^\rho_\mathsf{r}\}&=0.\\
\tag{Right-Full}\inf\{\|bd(\mathsf{s}(\rho(b)))-b\|:d\in\mathcal{D}^\rho_\mathsf{r}\}&=0.
\end{align*}
We call $\rho$ \emph{unit-full} if, for every $x\in\Gamma^0$, we have $d\in\mathcal{D}^\rho_\mathsf{r}$ with $d(x)\in B^0$.
\end{dfn}

\begin{prp}
If $\rho$ is left-full then
\begin{equation}\label{Lrho}
\mathcal{L}^\rho\mathcal{C}_\mathsf{r}^\rho\subseteq\mathcal{C}_\mathsf{r}^\rho\subseteq\mathcal{L}^\rho=\mathrm{cl}_\mathsf{b}(\mathcal{L}^\rho)\subseteq\mathcal{S}^\rho_\mathsf{b}.
\end{equation}
\end{prp}

\begin{proof}
Assume $\rho$ is left-full and say we have $l\in\mathcal{L}^\rho$ with $\|l\|_\mathsf{b}=\infty$.  This means we have $(f_n)\subseteq\mathcal{F}^\rho$ with $\|f_n\|_2=1$ and $\|lf_n\|_2>2^n$, for all $n\in\mathbb{N}$.  As $\rho$ is left-full, $d_\lambda f\rightarrow f$ for any approximate unit $(d_\lambda)$ for $\mathcal{D}^\rho_\mathsf{r}$ and any $f\in\mathcal{F}^\rho$.  It follows that we have $(d_n)\in\mathcal{D}^{\rho1}_+$ with $\|l\sqrt{d_n}f_n\|_2>2^n$ and hence $\|l\sqrt{d_n}\|_b>2^n$, for all $n\in\mathbb{N}$.  Letting $d=\sum_{n\in\mathbb{N}}2^{-n}d_n\in\mathcal{D}^\rho_\mathsf{r}$, \autoref{StarSquareInequality} yields
\[\|l\sqrt{d}\|_\mathsf{b}\geq2^{-n/2}\|l\sqrt{d_n}\|_\mathsf{b}>2^{n/2}\rightarrow\infty,\]
contradicting the fact that $l\sqrt{d}\in\mathcal{C}_\mathsf{r}^\rho\subseteq\mathcal{S}_\mathsf{b}^\rho$.  This shows that $\mathcal{L}^\rho\subseteq\mathcal{S}^\rho_\mathsf{b}$.

To see that $\mathcal{L}^\rho$ is $\mathsf{b}$-closed, take any $(l_n)\subseteq\mathcal{L}^\rho$ with $\mathsf{b}$-limit $l\in\mathcal{S}^\rho_\mathsf{b}$.  For any $d\in\mathcal{D}^\rho_\mathsf{r}$, we see that $\|l_nd-ld\|_\mathsf{b}\leq\|l_n-l\|_\mathsf{b}\|d\|_\mathsf{b}\rightarrow0$.  Thus $ld\in\mathrm{cl}_\mathsf{b}(\mathcal{C}^\rho_\mathsf{r})=\mathcal{C}^\rho_\mathsf{r}$ and hence $l\in\mathcal{L}^\rho$, as required.

Finally, to see that $\mathcal{L}^\rho\mathcal{C}_\mathsf{r}^\rho\subseteq\mathcal{C}_\mathsf{r}^\rho$, take $l\in\mathcal{L}^\rho$ and $a\in\mathcal{C}_\mathsf{r}^\rho$.  First note $la$ is defined, by \autoref{a*b}.  Any approximate unit $(d_\lambda)$ for $\mathcal{D}^\rho_\mathsf{r}\supseteq\mathcal{C}^{\rho_0}_\mathsf{r}$ will be an approximate unit for $\mathcal{C}^\rho_\mathsf{r}$ so $la$ is the $\mathsf{b}$-limit of $ld_\lambda a\in\mathcal{C}^\rho_\mathsf{r}\mathcal{C}^\rho_\mathsf{r}\subseteq\mathcal{C}^\rho_\mathsf{r}$ and hence $la\in\mathcal{C}_\mathsf{r}^\rho$ too.
\end{proof}

A \emph{left multiplier} of a C*-algebra $A$ is a right $A$-module map $l:A\rightarrow A$, i.e. a linear map such that $l(ab)=l(a)b$, for all $a,b\in A$.  These are always bounded and form a Banach algebra $\mathsf{L}(A)$ under composition (see \cite[3.12]{Pedersen2018}).  By \autoref{AssociativeProduct2} and the above result, every $l\in\mathcal{L}^\rho$ defines a left multiplier $l^\mathsf{L}$ of $\mathcal{C}_\mathsf{r}^\rho$ where
\[l^\mathsf{L}(a)=la.\]
Conversely, as long as $\rho$ is unit-full, every left-multiplier of $\mathcal{C}_\mathsf{r}^\rho$ will be of this form.

\begin{thm}
If $\rho$ is unit-full then $lk$ is a well-defined element of $\mathcal{L}^\rho$, for all $l,k\in\mathcal{L}^\rho$, and then $l\mapsto l^\mathsf{L}$ is a Banach algebra isomorphism from $\mathcal{L}^\rho$ onto $\mathsf{L}(\mathcal{C}^\rho_\mathsf{r})$.
\end{thm}

\begin{proof}
Assume $\rho$ is unit-full and take $l,k\in\mathcal{L}^\rho$.  For each $x\in\Gamma^0$, we have $d\in\mathcal{D}^\rho_\mathsf{r}$ with $d(x)\in B^0$ and hence $k_{\Gamma x}=kd_x=(kd)_{\Gamma x}$.  Then $lk_{\Gamma x}=l(kd)_{\Gamma x}=(l(kd))_{\Gamma x}$ is defined, by \autoref{a*b}, as $l\in\mathcal{L}^\rho\subseteq\mathcal{S}^\rho_\mathsf{b}$ and $kd\in\mathcal{C}^\rho_\mathsf{r}$.  As $x\in\Gamma^0$ was arbitrary, this shows that $lk$ is defined.  For any $d\in\mathcal{D}^\rho_\mathsf{r}$, we see that $(lk)d=l(kd)\in\mathcal{L}^\rho\mathcal{C}^\rho_\mathsf{r}\subseteq\mathcal{C}^\rho_\mathsf{r}$ and hence $lk\in\mathcal{L}^\rho$.  Arguing just like in the proof of \autoref{AssociativeProduct}, we then see that $(kl)a=k(la)$, for all $a\in\mathcal{C}^\rho_\mathsf{r}$, and hence $(lk)^\mathsf{L}=l^\mathsf{L}\circ k^\mathsf{L}$.  This shows that the map $l\mapsto l^\mathsf{L}$ is an algebra homomorphism on $\mathcal{L}^\rho$.  Moreover, $\|l^\mathsf{L}(a)\|_\mathsf{b}=\|la\|_\mathsf{b}\leq\|l\|_\mathsf{b}\|a\|_\mathsf{b}$ so $\|l^\mathsf{L}\|\leq\|l\|_\mathsf{b}$, for all $l\in\mathcal{L}^\rho$, i.e. the map $l\mapsto l^\mathsf{L}$ is also contractive.  It is also injective, as $l^\mathsf{L}=k^\mathsf{L}$ implies that $l_{\Gamma x}=l^\mathsf{L}(d)_{\Gamma x}=k^\mathsf{L}(d)_{\Gamma x}=k_{\Gamma x}$, for any $x\in\Gamma^0$ and $d\in\mathcal{D}^\rho_\mathsf{r}$ with $d(x)\in B^0$, and hence $l=k$.

Take a left multiplier $l:A\rightarrow A$ and let $(d_\lambda)$ be an approximate unit for $\mathcal{D}^\rho_\mathsf{r}$ and hence $\mathcal{C}^\rho_\mathsf{r}$.  For each $x\in\Gamma^0$ and $d\in\mathcal{D}^\rho_\mathsf{r}$ with $d(x)\in B^0$, note that
\[l(d)_{\Gamma x}=\lim\nolimits_\mathsf{b}l(d_\lambda d)_{\Gamma x}=\lim\nolimits_\mathsf{b}(l(d_\lambda)d)_{\Gamma x}=\lim\nolimits_\mathsf{b}l(d_\lambda)d_x=\lim\nolimits_\mathsf{b}l(d_\lambda)_{\Gamma x}.\]
Thus we can define $l^\rho\in\mathcal{S}^\rho$ by $l^\rho_{\Gamma x}=l(d)_{\Gamma x}=\lim_\mathsf{b}l(d_\lambda)_{\Gamma x}$, for all $x\in\Gamma^0$.  For any $d\in\mathcal{D}^\rho_\mathsf{r}$, we then again see that $l(d)_{\Gamma x}=\lim_\mathsf{b}l(d_\lambda)_{\Gamma x}d_x=l^\rho d_x=(l^\rho d)_{\Gamma x}$ and hence $l^\rho d=l(d)\in\mathcal{C}^\rho_\mathsf{r}$, showing that $l^\rho\in\mathcal{L}^\rho$.  Moreover, for all $f\in\mathcal{F}^\rho$, we see that $l^\rho f=\lim_\mathsf{b}l(d_\lambda)f$ and hence $\|l^\rho f\|_2\leq\sup\|l(d_\lambda)\|_\mathsf{b}\|f\|_2\leq\|l\|\|f\|_2$ so $\|l^\rho\|_\mathsf{b}\leq\|l\|$.

Next note that, for any $c\in\mathcal{C}^\rho_\mathsf{r}$ such that $\mathrm{supp}(c)$ is a slice and any $y\in\Gamma^0$, we have $x\in\Gamma^0$ such that $(ac)_{\Gamma y}=ac_{\Gamma y}=a_{\Gamma x}c$, for all $a\in\mathcal{S}^\rho$, and hence
\[l(c)_{\Gamma y}=\lim\nolimits_\mathsf{b}l(d_\lambda c)_{\Gamma y}=\lim\nolimits_\mathsf{b}l(d_\lambda)c_{\Gamma y}=\lim\nolimits_\mathsf{b}l(d_\lambda)_{\Gamma x}c=l^\rho_{\Gamma x}c=l^\rho c_{\Gamma y}.\]
Thus $l(c)=l^\rho c$, for all such $c$.  This then extends to linear combinations and their $\mathsf{b}$-limits (as we have already shown that $l^\rho\in\mathcal{S}^\rho_\mathsf{b}$) and hence to all $c\in\mathcal{C}^\rho_\mathsf{b}$, thus showing that $l=l^{\rho\mathsf{L}}$.  This completes the proof that $l\mapsto l^\mathsf{L}$ is indeed a Banach algebra isomorphism of $\mathcal{L}^\rho$ onto $\mathsf{L}(\mathcal{C}^\rho_\mathsf{r})$.
\end{proof}

When $\rho[B^0]=\Gamma^0$ is locally compact and Hausdorff, $\rho$ is automatically unit-full (and hence left-full and right-full), thanks to \cite[Corollary 2.10]{Lazar2018}.  When $\rho$ is also categorical (see \autoref{CategoricalChars}), we can characterise the left ($\rho$-)multipliers as those $\mathsf{b}$-bounded sections that look locally like elements of the reduced C*-algebra.

\begin{prp}\label{LeftLocallyReduced}
If $\rho$ is categorical and $\Gamma^0$ is locally compact and Hausdorff then
\begin{align*}
\mathcal{L}^\rho&=\{a\in\mathcal{S}^\rho_\mathsf{b}:\forall\text{ compact }K\subseteq\Gamma^0\ \exists c\in\mathcal{C}^\rho_\mathsf{r}\ (a_{\Gamma K}=c_{\Gamma K})\}\\
&=\{a\in\mathcal{S}^\rho_\mathsf{b}:\forall x\in\Gamma^0\ \exists\text{ open }O\ni x\ \exists c\in\mathcal{C}^\rho_\mathsf{r}\ (a_{\Gamma O}=c_{\Gamma O})\}.
\end{align*}
\end{prp}

\begin{proof}
Take any $a\in\mathcal{L}^\rho\subseteq\mathcal{S}^\rho_\mathsf{b}$ and compact $K\subseteq\Gamma^0$.  As $\Gamma^0$ is locally compact and Hausdorff, Urysohn's lemma yields continuous $f:\Gamma^0\rightarrow[0,1]$ with $L=\mathrm{cl}(\mathrm{supp}(f))\cap\Gamma^0$ compact and $f(x)=1$, for all $x\in K$.  As $\rho$ is categorical, \autoref{CategoricalChars} (3) applies and we then have $c\in\mathcal{C}^\rho(L)\subseteq\mathcal{C}^\rho_\mathsf{r}$ with $c(x)=f(x)1_x$, for all $x\in\Gamma^0$.  In particular, $c(x)=1_x$, for all $x\in K$, and hence $a_{\Gamma K}=(ac)_{\Gamma K}$ and $ac\in\mathcal{C}^\rho_\mathsf{r}$, as $a\in\mathcal{L}^\rho$.  This proves the $\subseteq$ part of the first line.

For the reverse inclusion, take $a\in\mathcal{S}^\rho_\mathsf{b}$ such that, for all compact $K\subseteq\Gamma^0$, we have $c\in\mathcal{C}^\rho_\mathsf{r}$ with $a_{\Gamma K}=c_{\Gamma K}$.  In particular, for any compact $K\subseteq\Gamma$ and $b\in\mathcal{C}^\rho(K)$, we see that $\mathsf{r}[K]$ is compact so we have $c\in\mathcal{C}^\rho_\mathsf{r}$ with $a_{\Gamma\mathsf{r}[K]}=c_{\Gamma\mathsf{r}[K]}$ and hence $ab=a_{\Gamma\mathsf{r}[K]}b=c_{\Gamma\mathsf{r}[K]}b=cb\in\mathcal{C}^\rho_\mathsf{r}$.  This immediately extends to linear combinations, showing that $a\mathcal{C}^\rho_\mathsf{c}\subseteq\mathcal{C}^\rho_\mathsf{r}$.  As $a$ is $\mathsf{b}$-bounded, this further extends to $\mathsf{b}$-limits and hence $a\mathcal{D}^\rho_\mathsf{r}\subseteq a\mathcal{C}^\rho_\mathsf{r}\subseteq\mathcal{C}^\rho_\mathsf{r}$, showing that $a\in\mathcal{L}^\rho$, completing the proof of the first line.

The $\subseteq$ of the second line is immediate from fact that each $x\in\Gamma^0$ has a compact neighbourhood.  Conversely, say we have $a\in\mathcal{S}^\rho_\mathsf{b}$ such that, for all $x\in\Gamma^0$, we have open neighbourhood $O\subseteq\Gamma^0$ and $c\in\mathcal{C}^\rho_\mathsf{r}$ with $a_{\Gamma O}=c_{\Gamma O}$.  Any compact $K$ can be covered by finitely many such $O$'s.  Using a subordinate partition of unity, we can then combine the corresponding $c$'s into a single $c\in\mathcal{C}^\rho_\mathsf{r}$ with $a_{\Gamma K}=c_{\Gamma K}$, thus completing the proof the second equality.
\end{proof}

We can also consider \emph{right multipliers} of a C*-algebra $A$, i.e. linear $r:A\rightarrow A$ with $r(ab)=ar(b)$, for all $a,b\in A$.  The right multiplier algebra $\mathsf{R}(A)$ is anti-isomorphic to the left multiplier algebra $\mathsf{L}(A)$ via $m\mapsto m^*$, where $m^*(a)=m(a^*)^*$.

We also have \emph{right $\rho$-multipliers} given by
\[\mathcal{R}^\rho=\{a\in\mathcal{S}^\rho:\forall z\in\mathcal{D}^\rho_\mathsf{r}\ (za\in\mathcal{C}_\mathsf{r}^\rho)\},\]
which again correspond bijectively to left $\rho$-multipliers under the map $a\mapsto a^*$.  As above, any $r\in\mathcal{R}^\rho$ defines a right multiplier $r^\mathsf{R}$ of $\mathcal{C}^\rho_\mathsf{r}$ given by $r^\mathsf{R}(a)=ar$.  Noting that $l^{\mathsf{L}*}=l^{*\mathsf{R}}$, for all $l\in\mathcal{L}^\rho$, we can then obtain dual results for right ($\rho$-)multipliers directly from the above results for left ($\rho$-)multipliers.

A \emph{multiplier} of a C*-algebra $A$ is a pair $(r,l)\in\mathsf{R}(A)\times\mathsf{L}(A)$ which satisfies $r(a)b=al(b)$, for all $a,b\in A$.  These form a C*-algebra denoted by $\mathsf{M}(A)$ where $(q,k)(r,l)=(r\circ q,k\circ l)$, $(r,l)^*=(l^*,r^*)$ and $\|(r,l)\|=\|r\|=\|l\|$.

Likewise, we have the \emph{$\rho$-multipliers} given by
\[\mathcal{M}^\rho=\mathcal{R}^\rho\cap\mathcal{L}^\rho=\{a\in\mathcal{S}^\rho:\forall z\in\mathcal{D}^\rho_\mathsf{r}\ (za,az\in\mathcal{C}_\mathsf{r}^\rho)\}.\]

\begin{cor}
If $\rho$ is unit-full then $\mathcal{M}^\rho$ is isomorphic to $\mathsf{M}(\mathcal{C}^\rho_\mathsf{r})$ via the map
\[m\mapsto(m^\mathsf{R},m^\mathsf{L}).\]
\end{cor}
 
\begin{proof}
If $\rho$ is unit-full then we already know that every multiplier is of the form $(r^\mathsf{R},l^\mathsf{L})$, for some $r\in\mathcal{R}^\rho$ and $l\in\mathcal{L}^\rho$.  It only remains to show that $r=l$.  To see this, take any $\gamma\in\Gamma$.  For any $c,d\in\mathcal{D}^\rho_\mathsf{r}$ with $c(\mathsf{r}(\gamma)),d(\mathsf{s}(\gamma))\in B^0$, we then see that
\[r(\gamma)=crd(\gamma)=(r^\mathsf{R}(c)d)(\gamma)=(cl^\mathsf{L}(d))(\gamma)=cld(\gamma)=l(\gamma).\]
As $\gamma$ was arbitrary, this shows that $l=r$.
\end{proof}

\subsection{The Essential C*-Algebra}

The essential C*-algebra is a quotient of the reduced C*-algebra designed to `mod out' any discontinuities.  Exel and Pitts first defined such a quotient in \cite{ExelPitts2022} for topologically principal groupoids, which was then extended to Fell bundles over more general \'etale groupoids by Kwa\'sniewski and Meyer in \cite{KwasniewskiMeyer2021}.  We basically follow the Kwa\'sniewski-Meyer approach, although we avoid the need to consider any kind of generalised expectations because we are already dealing with concrete sections of the Fell bundle.

At first it will again be convenient to consider more general classes of sections defined as follows.
\begin{align*}
\mathcal{S}^\rho_\sigma&=\{a\in\mathcal{S}^\rho_\mathsf{b}:\mathrm{supp}(a)\subseteq K,\text{ for some $\sigma$-bicompact }K\subseteq\Gamma\}.\\
\mathcal{S}^\rho_\mathsf{m}&=\{a\in\mathcal{S}^\rho_\sigma:\mathrm{supp}(a)\text{ is meagre}\}.\\
\mathcal{B}^\rho_\sigma&=\mathcal{B}^\rho\cap\mathcal{S}^\rho_\sigma.\\
\mathcal{B}^\rho_\mathsf{m}&=\mathcal{B}^\rho\cap\mathcal{S}^\rho_\mathsf{m}.
\end{align*}

\begin{prp}
$\mathcal{B}^\rho_\sigma$ is a C*-algebra and $\mathcal{S}^\rho_\mathsf{m}$ is a Banach *-bimodule over $\mathcal{B}^\rho_\sigma$.
\end{prp}

\begin{proof}
For any $a,b\in\mathcal{S}^\rho$ and $\lambda\in\mathbb{C}$, note that $\mathrm{supp}(a+b)\subseteq\mathrm{supp}(a)\cup\mathrm{supp}(b)$, $\mathrm{supp}(\lambda a)\subseteq\mathrm{supp}(a)=\mathrm{supp}(a^*)^{-1}$ and $\mathrm{supp}(ab)\subseteq\mathrm{supp}(a)\mathrm{supp}(b)$.  As $\sigma$-bicompact subsets are closed under pairwise unions, inverses and products (by \autoref{CompactProducts}), it follows that $\mathcal{B}^\rho_\sigma$ is a *-subalgebra of $\mathcal{B}^\rho$.  Also, whenever $a=\lim_\mathsf{b}a_n$ in $\mathcal{S}^\rho_\mathsf{b}$, we see that $\mathrm{supp}(a)\subseteq\bigcup\mathrm{supp}(a_n)$.  As $\sigma$-bicompact subsets are closed under countable unions, it follows that $\mathcal{B}^\rho_\sigma$ is $\mathsf{b}$-closed too and hence a C*-subalgebra of $\mathcal{B}^\rho$.

Likewise, we see that $\mathcal{S}^\rho_\mathsf{m}$ is a $\mathsf{b}$-closed *-invariant subspace of $\mathcal{S}^\rho_\sigma$.  To see that it is also closed under products with $\mathcal{B}^\rho_\sigma$, take any $a\in\mathcal{S}^\rho_\mathsf{m}$ and $b\in\mathcal{B}^\rho_\sigma$.  Then $\mathsf{r}(\mathrm{supp}(ab))\subseteq\mathsf{r}(\mathrm{supp}(a))$ is meagre, by \eqref{sMeagre}, and hence $\mathrm{supp}(ab)$ is meagre, again by \eqref{sMeagre}, so $ab\in\mathcal{S}^\rho_\mathsf{m}$.  Thus $\mathcal{S}^\rho_\mathsf{m}$ is a sub-*-bimodule of $\mathcal{S}^\rho_\sigma$ over $\mathcal{B}^\rho_\sigma$.
\end{proof}

It follows that $\mathcal{B}^\rho_\mathsf{m}$ a closed ideal of $\mathcal{B}^\rho_\sigma$.  Taking the quotient, we obtain the \emph{essentially big C*-algebra}, which we denote by
\[\mathcal{B}^\rho_\mathsf{e}=\mathcal{B}^\rho_\sigma/\mathcal{B}^\rho_\mathsf{m}.\]
Let $\pi_\mathsf{e}:\mathcal{B}_\sigma\rightarrow\mathcal{B}_\mathsf{e}$ be the canonical homomorphism and let $\|\cdot\|_\mathsf{e}$ denote the seminorm on $\mathcal{B}_\sigma$ coming from the quotient norm on $\mathcal{B}^\rho_\mathsf{e}$, i.e.
\[\|a\|_\mathsf{e}=\|\pi_\mathsf{e}(a)\|=\inf_{b\in\mathcal{B}^\rho_\mathsf{m}}\|a-b\|_\mathsf{b}.\]
In particular, for any $G\subseteq\Gamma$ such that $\mathrm{supp}(a)\setminus G$ is meagre,
\[\|a\|_\mathsf{e}\leq\|a-a_{\mathrm{supp}(a)\setminus G}\|_\mathsf{b}=\|a_G\|_\mathsf{b}.\]
We can even obtain equality here for certain special $G$.

\begin{prp}\label{ComeagreEssentialNorm}
For every $a\in\mathcal{B}^\rho_\sigma$ and meagre $M\subseteq\Gamma$, we have $\sigma$-bicompact $K\subseteq\Gamma$ and a subgroupoid $G\subseteq K$ such that $G\cap M=\emptyset$, $\mathrm{supp}(a)\setminus G$ is meagre and
\[\|a\|_\mathsf{e}=\|a_G\|_\mathsf{b}.\]
\end{prp}

\begin{proof}
Take $(b_n)\subseteq\mathcal{B}^\rho_\mathsf{m}$ with $\|a-b_n\|_\mathsf{b}\rightarrow\|a\|_\mathsf{e}$ and let $N=\bigcup_{n\in\mathbb{N}}\mathrm{supp}(b_n)$.  By \autoref{ComeagreSubgroupoids}, we have $\sigma$-bicompact $L\subseteq\Gamma$ and subgroupoids $G,H\subseteq L$ with $H$ meagre, $\mathrm{supp}(a)\cup N\subseteq G\cup H$, $G\cap(M\cup N)=\emptyset$,  and $(\mathsf{s}[G]\cup\mathsf{r}[G])\cap(\mathsf{s}[H]\cup\mathsf{r}[H])=\emptyset$.  Note $\mathrm{supp}(a)\setminus G\subseteq H$ is also meagre and hence $\|a\|_\mathsf{e}\leq\|a_G\|_\mathsf{b}$.  Moreover, $a_G=a_{\Gamma G}$ because $\mathsf{s}[\mathrm{supp}(a)\setminus G]\subseteq\mathsf{s}[H]\subseteq\Gamma^0\setminus\mathsf{s}[G]$.  And for all $n\in\mathbb{N}$, $a_{\Gamma G}=(a-b_n)_{\Gamma G}$ because $N\subseteq H$ so $\mathsf{s}[\mathrm{supp}(b_n)]\subseteq\mathsf{s}[N]\subseteq\mathsf{s}[H]\subseteq\Gamma^0\setminus\mathsf{s}[G]$.  It follows that
\[\|a_G\|_\mathsf{b}=\|a_{\Gamma G}\|_\mathsf{b}=\|(a-b_n)_{\Gamma G}\|_\mathsf{b}\leq\|a-b_n\|_\mathsf{b},\]
for all $n\in\mathbb{N}$, by \eqref{bRestriction} (noting $\Gamma G=\Gamma\mathsf{s}[G]$), and hence $\|a_G\|_\mathsf{b}=\|a\|_\mathsf{e}$.
\end{proof}

The \emph{essential C*-algebra} of $\rho$ is the quotient of the reduced C*-algebra $\mathcal{C}^\rho_\mathsf{r}$ by the \emph{singular ideal}  $\mathcal{C}^\rho_\mathsf{m}=\mathcal{B}^\rho_\mathsf{m}\cap\mathcal{C}^\rho_\mathsf{r}$, which we denote by
\[\mathcal{C}^\rho_\mathsf{e}=\mathcal{C}^\rho_\mathsf{r}/\mathcal{C}^\rho_\mathsf{m}\approx\pi_\mathsf{e}[\mathcal{C}^\rho_\mathsf{r}].\]
The following shows that we can view any element of $\mathcal{C}^\rho_\mathsf{e}$ as a section $a$ on a subgroupoid of $\Gamma$ obtained by ignoring the points on which $a$ and $a_\infty$ are discontinuous (where again $a_\infty:\Gamma\rightarrow\mathbb{R}_+$ is the norm function defined by $a_\infty(\gamma)=\|a(\gamma)\|$).

\begin{thm}
For every $a\in\mathcal{C}^\rho_\mathsf{r}$, we have $\sigma$-bicompact $K\subseteq\Gamma$ and a subgroupoid $G\subseteq K$ contained in $\mathcal{C}(a)\cap\mathcal{C}(a_\infty)$ with $\mathrm{supp}(a)\setminus G$ meagre and $\|a\|_\mathsf{e}=\|a_G\|_\mathsf{b}$.
\end{thm}

\begin{proof}
By \autoref{ComeagreNormContinuity} and \autoref{ComeagreContinuity}, $\mathcal{C}(a_\infty)$ and $\mathcal{C}(a)$ are comeagre and hence so is their intersection.  In other words, $M=\Gamma\setminus(\mathcal{C}(a_\infty)\cap\mathcal{C}(a))$ is meagre so we can just apply \autoref{ComeagreEssentialNorm} to obtain the result.
\end{proof}

\section{Morphisms}

Morphisms of section algebras/modules of Fell bundles can come from morphisms of the total spaces or base groupoids, or indeed a combination of the two.

\subsection{Base Morphisms}

First we consider groupoid morphisms.

\begin{dfn}
A functor $\phi:\Gamma\rightarrow\Gamma'$ between groupoids $\Gamma$ and $\Gamma'$ is \emph{star-bijective} if, whenever $\phi(x)=\mathsf{s}(\gamma')$, there is precisely one $\gamma\in\Gamma$ with $\mathsf{s}(\gamma)=x$ and $\phi(\gamma)=\gamma'$.
\end{dfn}

More symbolically, this can be written as
\[\tag{Star-Bijective}\phi(x)=\mathsf{s}(\gamma')\qquad\Rightarrow\qquad|\mathsf{s}^{-1}\{x\}\cap\phi^{-1}\{\gamma'\}|=1.\]
See \cite[\S2]{Bice2022} for several equivalent characterisations of star-bijectivity.

Recall that a continuous map $\phi:X\rightarrow Y$ between Hausdorff topological spaces $X$ and $Y$ is \emph{proper} if preimages of compact subsets of $Y$ are compact in $X$.

\begin{dfn}
An \emph{\'etale morphism} is a proper continuous star-bijective functor $\phi:\Gamma\rightarrow\Gamma'$ from an \'etale groupoid $\Gamma$ to another \'etale groupoid $\Gamma'$.
\end{dfn}

Given a Fell bundle $\rho:B\twoheadrightarrow\Gamma'$ and an \'etale morphism $\phi:\Gamma\rightarrow\Gamma'$, let
\[\phi^\rho B=\{(\gamma,b)\in\Gamma\times B:\phi(\gamma)=\rho(b)\}.\]
We consider $\phi^\rho B$ as a topological subspace of $\Gamma\times B$ with operations
\begin{align*}
\lambda(\gamma,b)&=(\gamma,\lambda b).\\
(\gamma,a)+(\gamma,a)&=(\gamma,a+b).\\
(\beta,a)(\gamma,b)&=(\beta\gamma,ab).\\
(\gamma,b)^*&=(\gamma^{-1},b^*).\\
\|(\gamma,b)\|&=\|b\|.\\
\rho_\phi(\gamma,b)&=\gamma.
\end{align*}
Then $\rho_\phi:\phi^\rho B\twoheadrightarrow\Gamma$ is also a Fell bundle known as the \emph{pullback bundle} of $\rho$ induced by $\phi$.  The pullback-section map $\phi^\rho:\mathcal{S}^\rho\rightarrow\mathcal{S}^{\rho_\phi}$ is then given by
\[\phi^\rho(a)(\gamma)=(\gamma,a(\phi(\gamma))).\]

\begin{prp}\label{phirho}
The map $\phi^\rho$ restricts to a C*-algebra homomorphism from $\mathcal{B}^\rho$ to $\mathcal{B}^{\rho_\phi}$ which in turn restricts to a C*-algebra homomorphism from $\mathcal{C}^\rho_\mathsf{r}$ to $\mathcal{C}^{\rho_\phi}_\mathsf{r}$.
\end{prp}

\begin{proof}
Certainly $\phi^\rho$ is a *-preserving vector space homomorphism on $\mathcal{S}^\rho$.  To see that it also preserves products whenever they are defined, take $a,b\in\mathcal{S}^\rho$ such that $ab$ is defined.  As $\phi$ is star-bijective, whenever $\phi(\gamma)=\alpha'\beta'$, there exist unique $\alpha,\beta\in\Gamma$ with $\phi(\alpha)=\alpha'$, $\phi(\beta)=\beta'$ and $\gamma=\alpha\beta$ (first take $\beta\in\Gamma$ with $\phi(\beta)=\beta'$ and $\mathsf{s}(\beta)=\mathsf{s}(\gamma)$ and then let $\alpha=\gamma\beta^{-1}$).  It follows that, for any $\gamma\in\Gamma$,
\[ab(\phi(\gamma))=\sum_{\phi(\gamma)=\alpha'\beta'}a(\alpha')b(\beta')=\sum_{\gamma=\alpha\beta}a(\phi(\alpha))b(\phi(\beta)).\]
As fibres of $\rho_\phi$ are linearly isometric to the corresponding fibres of $\rho$, infinite sums that are valid in one are also valid in the other and hence
\begin{align*}
\phi^\rho(ab)(\gamma)&=(\gamma,ab(\phi(\gamma)))\\
&=(\gamma,\sum_{\gamma=\alpha\beta}a(\phi(\alpha))b(\phi(\beta)))\\
&=\sum_{\gamma=\alpha\beta}(\alpha,a(\phi(\alpha)))(\beta,b(\phi(\beta)))\\
&=\sum_{\gamma=\alpha\beta}\phi^\rho(a)(\alpha)\phi^\rho(b)(\beta)\\
&=(\phi^\rho(a)\phi^\rho(b))(\gamma)
\end{align*}
This shows that indeed $\phi^\rho(ab)=\phi^\rho(a)\phi^\rho(b)$.

We immediately see that $\phi^\rho$ is $\infty$-contractive on $\mathcal{S}^\rho$, as
\[\|\phi^\rho(a)\|_\infty=\|a_{\phi[\Gamma]}\|_\infty\leq\|a\|_\infty.\]
As $\phi$ respects all the various operations, it likewise follows that
\[\|\phi^\rho(a)\|_2=\|a_{\phi[\Gamma]}\|_2\leq\|a\|_2,\]
i.e. $\phi^\rho$ is $2$-contractive on $\mathcal{S}^\rho$.  As $\phi$ is star-bijective, $\phi[\Gamma]=\Gamma\mathsf{s}[\phi[\Gamma]]$ and hence
\[\|\phi^\rho(a)\|_\mathsf{b}=\|a_{\phi[\Gamma]}\|_\mathsf{b}\leq\|a\|_\mathsf{b},\]
i.e. $\phi^\rho$ is also $\mathsf{b}$-contractive on $\mathcal{S}^\rho$.  It follows that $\phi^\rho[\mathcal{S}^\rho_\mathsf{b}]\subseteq\mathcal{S}^{\rho_\phi}_\mathsf{b}$ and therefore $\phi^\rho[\mathcal{B}^\rho]\subseteq\mathcal{B}^{\rho_\phi}$, as we already saw that $\phi^\rho(a)^*\phi^\rho(a)$ is defined whenever $a^*a$ is.  Thus $\phi^\rho$ does indeed restrict to a C*-algebra homomorphism from $\mathcal{B}^\rho$ to $\mathcal{B}^{\rho_\phi}$.

As $\phi$ is continuous, $\phi^\rho[\mathcal{C}^\rho(K)]\subseteq\mathcal{C}^{\rho_\phi}(\phi^{-1}[K])$, for any $K\subseteq\Gamma'$.  As $\phi$ is also proper and star-bijective, $\phi^{-1}[K]$ is bicompact whenever $K$ is and so $\phi^\rho[\mathcal{C}^\rho_\mathsf{c}]\subseteq\mathcal{C}^{\rho_\phi}_\mathsf{c}$.  As $\phi^\rho$ is also $\mathsf{b}$-contractive, $\phi^\rho[\mathcal{C}_\mathsf{r}^\rho]\subseteq\mathcal{C}_\mathsf{r}^{\rho_\phi}$ and hence $\phi^\rho$ does indeed restrict to a C*-algebra homomorphism from $\mathcal{C}_\mathsf{r}^\rho$ to $\mathcal{C}_\mathsf{r}^{\rho_\phi}$.
\end{proof}

\subsection{Fibre Morphisms}

Next we consider morphisms of fibres.

First let us call a subgroupoid $\Gamma$ of a \'etale groupoid $\Gamma'$ \emph{nice} if it is open and, moreover, every compact closed subset of $\Gamma^0$ is still closed in $\Gamma'^0$ (which is automatic if $\Gamma'^0$ is Hausdorff, for example).  Then every bicompact $K\subseteq\Gamma$ is bicompact in $\Gamma'$.

\begin{dfn}
Let $\rho:B\rightarrow\Gamma$ and $\rho':B'\rightarrow\Gamma'$ be Fell bundles where $\Gamma$ is a nice subgroupoid of $\Gamma'$.  A \emph{fibre morphism} from $\rho$ to $\rho'$ is a continuous fibre-wise linear *-homomorphism $\beta:B\rightarrow B'$, i.e. when $(a,b)\in B^2$, $\lambda\in\mathbb{C}$ and $\rho(b)=\rho(c)$,
\begin{align*}
\rho'(\beta(a))&=\rho(a).\\
\beta(a^*)&=\beta(a)^*.\\
\beta(ab)&=\beta(a)\beta(b)\\
\beta(\lambda a)&=\lambda\beta(a).\\
\beta(b+c)&=\beta(b)+\beta(c).
\end{align*}
\end{dfn}

\begin{prp}
Every fibre morphism is contractive on fibres.
\end{prp}

\begin{proof}
Say $\beta:B\rightarrow B'$ is a bundle morphism from $\rho:B\rightarrow\Gamma$ to $\rho':B'\rightarrow\Gamma'$.  In particular, for any $x\in\Gamma^0$, the restriction of $\beta$ to $B_x$ is a C*-algebra homomorphism to $B'_x$ and is thus contractive.  For any $b\in B$ with $\mathsf{s}(\rho(b))=x$, it follows that
\[\|\beta(b)\|^2=\|\beta(b)^*\beta(b)\|=\|\beta(b^*b)\|\leq\|b^*b\|=\|b\|^2.\]
Thus $\|\beta(b)\|\leq\|b\|$, showing that $\beta$ is contractive.
\end{proof}

If $\beta:B\rightarrow B'$ is a bundle morphism from $\rho:B\rightarrow\Gamma$ to $\rho':B'\rightarrow\Gamma'$, where $\Gamma$ is an open subgroupoid of $\Gamma'$, then we define $\beta_\rho^{\rho'}:\mathcal{S}^\rho\rightarrow\mathcal{S}^{\rho'}$ by
\begin{equation}\label{BetaDef}
\beta_\rho^{\rho'}(a)(\gamma')=\begin{cases}\beta(a(\gamma'))&\text{if }\gamma'\in\Gamma\\0_{\gamma'}&\text{if }\gamma'\in\Gamma'\setminus\Gamma.\end{cases}
\end{equation}

\begin{thm}\label{betarho}
The map $\beta_\rho^{\rho'}$ restricts to a C*-algebra homomorphism from $\mathcal{B}^\rho$ to $\mathcal{B}^{\rho'}$ which in turn restricts to a C*-algebra homomorphism from $\mathcal{C}^\rho_\mathsf{r}$ to $\mathcal{C}^{\rho'}_\mathsf{r}$.
\end{thm}

\begin{proof}
We immediately see that $\beta_\rho^{\rho'}$ is a *-preserving vector space homomorphism on $\mathcal{S}^\rho$.  Next note that, as $\beta$ is contractive, $\beta$ preserves infinite sums, i.e. if $\gamma\in\Gamma$, $(b_\lambda)_{\lambda\in\Lambda}\subseteq B_\gamma$ and $\sum_{\lambda\in\Lambda}b_\lambda$ is defined then $\beta(\sum_{\lambda\in\Lambda}b_\lambda)=\sum_{\lambda\in\Lambda}\beta(b_\lambda)$.  Thus if $a,b\in\mathcal{S}^\rho$ and $ab$ is defined then
\begin{align*}
\beta_\rho^{\rho'}\!(ab)(\gamma)&=\beta\Big(\sum_{\zeta\in\Gamma\gamma}a(\gamma\zeta^{-1})b(\zeta)\Big)\\
&=\sum_{\zeta\in\Gamma\gamma}\beta(a(\gamma\zeta^{-1}))\beta(b(\zeta))\\
&=\sum_{\zeta\in\Gamma'\gamma}\beta_\rho^{\rho'}\!(a)(\gamma\zeta^{-1})\beta_\rho^{\rho'}\!(b)(\zeta)\\
&=(\beta_\rho^{\rho'}\!(a)\beta_\rho^{\rho'}\!(b))(\gamma).
\end{align*}
In other words, $\beta_\rho^{\rho'}$ also preserves products whenever they are defined.

Also note that $\beta_\rho^{\rho'}$ respects restrictions, i.e. $\beta_\rho^{\rho'}\!(a_Y)=\beta_\rho^{\rho'}\!(a)_Y$, for any $Y\subseteq\Gamma$.  As $\beta$ is contractive, it follows immediately that $\beta_\rho^{\rho'}$ is $\infty$-contractive.  As $\beta$ preserves the relevant operations, $\beta_\rho^{\rho'}$ is then $2$-contractive, i.e. for all $a\in\mathcal{S}^\rho$,
\[\|\beta_\rho^{\rho'}\!(a)\|_2^2=\sup_{F\subset\Gamma}\|\beta_\rho^{\rho'}\!(a)^*\beta_\rho^{\rho'}\!(a)_F\|_\infty=\sup_{F\subset\Gamma}\|\beta_\rho^{\rho'}(a^*a_F)\|_\infty\leq\sup_{F\subset\Gamma}\|a^*a_F\|_\infty=\|a\|_2^2.\]

We further claim that
\begin{equation}\label{2ContractiveInverse}
\beta_\rho^{\rho'}[(\mathcal{H}^\rho)^1_2]=\beta_\rho^{\rho'}[\mathcal{H}^\rho]^1_2,
\end{equation}
where $H^1_2=\{h\in H:\|h\|_2\leq1\}$.  Indeed, $\beta_\rho^{\rho'}[(\mathcal{H}^\rho)^1_2]\subseteq\beta_\rho^{\rho'}[\mathcal{H}^\rho]^1_2$ is immediate from $2$-contractivity.  Conversely, take $h'\in\beta_\rho^{\rho'}[\mathcal{H}^\rho]^1_2$, so we have $h\in\mathcal{H}^\rho$ with $\beta_\rho^{\rho'}(h)=h'$.  For each $x\in\Gamma^0$, $h^*h(x)$ is a positive element of the C*-algebra $B_x$.  For each $n\in\mathbb{N}$, let $f_n$ be the function on $\mathbb{R}_+$ that maps $[0,1/n]$ linearly to $[0,1]$, has constant value $1$ on $[1/n,1]$ and takes any $r\geq1$ to $1/\sqrt{r}$.  Applying the continuous functional calculus, we then define a function $d_n\in\mathcal{D}^\rho$ of norm at most one by
\[d_n(x)=f_n(h^*h(x)).\]
As $|f_m(r)-f_n(r)|\leq1$, for all $r\in\mathbb{R}_+$, and $f_m(r)-f_n(r)=0$, for all $r\geq1/(m\wedge n)$,
\begin{align*}
\|h(d_m-d_n)\|_2^2&=\|\Phi((d_m-d_n)h^*h(d_m-d_n))\|_\infty\\
&=\sup_{x\in\Gamma^0}\|(f_m-f_n)^2(h^*h(x))h^*h(x)\|\\
&\leq1/(m\wedge n).
\end{align*}
Thus $hd_n$ is $2$-Cauchy and has a $2$-limit $g\in\mathcal{H}^\rho$.  As $f_n(r)^2r\leq1$, for all $r\in\mathbb{R}_+$,
\[\|hd_n\|_2^2=\sup_{x\in\Gamma^0}\|f_n(h^*h(x))^2h^*h(x)\|\leq1,\]
for all $n\in\mathbb{N}$, and hence $\|g\|_2\leq1$.  Also, for all $x\in\Gamma^0$, $\beta$ is a C*-algebra homomorphism from $B_x$ to $B'_x$ and so respects the continuous functional calculus.  As $(1-f_n)(r)^2r\leq1/n$, for all $r\leq1$, and $\|\beta(h^*h(x))\|\leq\|\beta_\rho^{\rho'}\!(h)\|_2^2=\|h'\|_2^2=1$,
\begin{align*}
\|h'-\beta_\rho^{\rho'}\!(hd_n)\|_2^2&=\|\beta_\rho^{\rho'}\!(h-hd_n)\|_2^2\\
&=\sup_{x\in\Gamma^0}\|\beta((h-hd_n)^*(h-hd_n)(x))\|\\
&=\sup_{x\in\Gamma^0}\|(1-f_n)(\beta(h^*h(x)))^2\beta(h^*h(x))\|\\
&\leq1/n\rightarrow0.
\end{align*}
Thus $h'=\beta_\rho^{\rho'}\!(g)$, as they are both the $2$-limit of $(\beta_\rho^{\rho'}\!(hd_n))$ (because $g$ is the $2$-limit of $(hd_n)$ and $\beta_\rho^{\rho'}$ is $2$-contractive).  This completes the proof of claim \eqref{2ContractiveInverse}.

For any $a\in\mathcal{B}^\rho$ and $\varepsilon>0$, \eqref{bC} yields $f'\in\beta_\rho^{\rho'}[\mathcal{F}^\rho]$ such that $\|f'\|_2=1$ and
\[\|\beta_\rho^{\rho'}(a)\|_\mathsf{b}\leq\|\beta_\rho^{\rho'}(a)f'\|_2+\varepsilon\]
By what we just proved, we have $f\in\mathcal{H}^\rho$ (in fact $f\in\mathcal{F}^\rho$) such that $\|f\|_2=1$ and $f'=\beta_\rho^{\rho'}(f)$.  As $\beta_\rho^{\rho'}$ is $2$-contractive, it follows that
\[\|\beta_\rho^{\rho'}(a)\|_\mathsf{b}\leq\|\beta_\rho^{\rho'}(af)\|_2+\varepsilon\leq\|af\|_2+\varepsilon\leq\|a\|_\mathsf{b}+\varepsilon.\]
This shows that $\beta_\rho^{\rho'}$ is $\mathsf{b}$-contractive on $\mathcal{B}^\rho$.  It follows that $\beta_\rho^{\rho'}[\mathcal{B}^\rho]\subseteq\mathcal{B}^{\rho'}$ and hence $\beta_\rho^{\rho'}$ does indeed restrict to a C*-algebra homomorphism from $\mathcal{B}^\rho$ to $\mathcal{B}^{\rho'}$.

Next note $\beta_\rho^{\rho'}[\mathcal{C}_\mathsf{c}^\rho]\subseteq\mathcal{C}_\mathsf{c}^{\rho'}$, as $\Gamma$ is nice and $\mathrm{supp}(\beta_\rho^{\rho'}(a))\subseteq\mathrm{supp}(a)$, for all $a\in\mathcal{S}^\rho$.  Thus $\beta_\rho^{\rho'}[\mathcal{C}_\mathsf{r}^\rho]=\beta_\rho^{\rho'}[\mathrm{cl}_\mathsf{b}(\mathcal{C}_\mathsf{c}^\rho)]\subseteq\mathrm{cl}_\mathsf{b}(\beta_\rho^{\rho'}[\mathcal{C}_\mathsf{c}^\rho])\subseteq\mathrm{cl}_\mathsf{b}(\mathcal{C}_\mathsf{c}^{\rho'})=\mathcal{C}_\mathsf{r}^{\rho'}$ and hence $\beta_\rho^{\rho'}$ further restricts to a C*-algebra homomorphism from $\mathcal{C}^\rho_\mathsf{r}$ to $\mathcal{C}^{\rho'}_\mathsf{r}$.
\end{proof}

\subsection{Fell Morphisms}

We now consider a category of \'etale-fibre morphism pairs between Fell bundles over \'etale groupoids.

\begin{dfn}
Let $\mathbf{Fell}$ denote quadruples $(\rho',\beta,\phi,\rho)$ where
\begin{enumerate}
\item $\rho:B\twoheadrightarrow\Gamma$ and $\rho':B'\twoheadrightarrow\Gamma'$ are Fell bundles.
\item $\phi$ is an \'etale morphism from a nice subgroupoid of $\Gamma'$ to $\Gamma$.
\item $\beta$ is a fibre morphism from $\rho_\phi$ to $\rho'$.
\end{enumerate}
When $(\rho',\beta,\phi,\rho)\in\mathbf{Fell}$, we say the pair $(\beta,\phi)$ is a \emph{Fell morphism} from $\rho$ to $\rho'$.
\end{dfn}

These are the natural extension to Fell bundles of the C*-bundle morphisms considered in \cite[Definition 3.1]{Kwasniewski2016} and \cite[Definition 4.3]{Varela1974}.

First we show that they do indeed turn Fell bundles into a category.

\begin{prp}
$\mathbf{Fell}$ forms a category under the product
\[(\rho'',\beta',\phi',\rho')(\rho',\beta,\phi,\rho)=(\rho'',\beta'\bullet\beta,\phi\circ\phi',\rho)\]
where $\beta'\bullet\beta=\beta'\bullet_{\phi'}\beta$ is the fibre morphism from $\rho_{\phi\circ\phi'}$ to $\rho''$ defined by
\[\beta'\bullet\beta(\gamma'',b)=\beta'(\gamma'',\beta(\phi'(\gamma''),b)).\]
\end{prp}

\begin{proof}
For associativity, take $(\rho''',\beta'',\phi'',\rho''),(\rho'',\beta',\phi',\rho'),(\rho',\beta,\phi,\rho)\in\mathbf{Fell}$ and note that, for all $\gamma'''\in\mathrm{dom}(\phi\circ\phi'\circ\phi'')=\phi''^{-1}[\phi'^{-1}[\phi^{-1}[\Gamma]]]$ and $b\in B$,
\begin{align*}
(\beta''\bullet(\beta'\bullet\beta))(\gamma''',b)&=\beta''(\gamma''',(\beta'\bullet\beta)(\phi''(\gamma'''),b))\\
&=\beta''(\gamma''',\beta'(\phi''(\gamma'''),\beta(\phi'(\phi''(\gamma''')),b)))\\
&=(\beta''\bullet\beta')(\gamma''',\beta(\phi'\circ\phi''(\gamma'''),b))\\
&=((\beta''\bullet\beta')\bullet\beta)(\gamma''',b).
\end{align*}
This shows that $\beta''\bullet(\beta'\bullet\beta)=(\beta''\bullet\beta')\bullet\beta$, as required.

We immediately see that every $(\rho',\beta,\phi,\rho)\in\mathbf{Fell}$ has a source unit $(\rho,\mathsf{p}_B,\mathrm{id}_\Gamma,\rho)$ and range unit $(\rho',\mathsf{p}_{B'},\mathrm{id}_{\Gamma'},\rho')$ satisfying \eqref{CategoryDef}, where $\mathrm{id}_\Gamma$ denotes the identity on $\Gamma$ and $\mathsf{p}_B:\Gamma\times B\rightarrow B$ denotes the right projection (where $\rho:B\twoheadrightarrow\Gamma$).  Thus $\mathbf{Fell}$ forms a category under the given product.
\end{proof}

Let $\mathbf{C^*}$ denote triples $(A',\pi,A)$ where $A$ and $A'$ are C*-algebras and $\pi:A\rightarrow A'$ is a C*-algebra homomorphism, which forms a category in the usual way where
\[(A'',\pi',A')(A',\pi,A)=(A',\pi'\circ\pi,A).\]

\begin{thm}\label{Afunctor}
We have a functor $\mathsf{A}:\mathbf{Fell}\rightarrow\mathbf{C^*}$ given by
\[\mathsf{A}(\rho',\beta,\phi,\rho)=(\mathcal{C}_\mathsf{r}^{\rho'},\beta_{\rho_\phi}^{\rho'}\!\circ\phi^\rho_\mathsf{r},\mathcal{C}_\mathsf{r}^\rho)\]
\end{thm}

\begin{proof}
By \autoref{phirho}, $\phi^\rho|_{\mathcal{C}_\mathsf{r}^\rho}$ is a C*-algebra homomorphism to $\mathcal{C}_\mathsf{r}^{\rho_\phi}$ and, by \autoref{betarho}, $\beta_{\rho_\phi}^{\rho'}|_{\mathcal{C}_\mathsf{r}^{\rho_\phi}}$ is a C*-algebra homomorphism to $\mathcal{C}_\mathsf{r}^{\rho'}$.   Thus their composition $\beta_{\rho_\phi}^{\rho'}\!\circ\phi^\rho$ is also C*-algebra homomorphism so $\mathsf{A}$ indeed maps $\mathbf{Fell}$ to $\mathbf{C^*}$.

Take $(\rho'',\beta',\phi',\rho'),(\rho',\beta,\phi,\rho)\in\mathbf{Fell}$, where $\rho:B\twoheadrightarrow\Gamma$.  Then $\mathsf{A}$ takes their product $(\rho'',\beta'\bullet\beta,\phi\circ\phi',\rho)$ to $(\mathcal{C}_\mathsf{r}^{\rho''},(\beta'\bullet\beta)_{\rho_{\phi\circ\phi'}}^{\rho'}\!\circ(\phi\circ\phi')^\rho,\mathcal{C}_\mathsf{r}^\rho)$.  For any $a\in\mathcal{C}_\mathsf{r}^\rho$ and $\gamma''\in\mathrm{dom}(\phi\circ\phi')=\phi'^{-1}[\phi^{-1}[\Gamma]]$, we see that
\begin{align*}
((\beta'\bullet\beta)_{\rho_{\phi\circ\phi'}}^{\rho'}\!\circ(\phi\circ\phi')^\rho)(a)(\gamma'')&=(\beta'\bullet\beta)_{\rho_{\phi\circ\phi'}}^{\rho'}\!((\phi\circ\phi')^\rho(a))(\gamma'')\\
&=(\beta'\bullet\beta)(\gamma'',a(\phi\circ\phi'(\gamma'')))\\
&=\beta'(\gamma'',\beta(\phi'(\gamma''),a(\phi(\phi'(\gamma'')))))\\
&=\beta'(\gamma'',\beta(\phi^\rho(a)(\phi'(\gamma''))))\\
&=\beta'(\gamma'',\beta_{\rho_\phi}^{\rho'}\!\circ\phi^\rho(a)(\phi'(\gamma'')))\\
&=\beta'(\phi'^{\rho'}\circ\beta_{\rho_\phi}^{\rho'}\!\circ\phi^\rho(a)(\gamma''))\\
&=(\beta'^{\rho''}_{\rho'_{\phi'}}\!\circ\phi'^{\rho'}\circ\beta_{\rho_\phi}^{\rho'}\!\circ\phi^\rho(a))(\gamma'').
\end{align*}
For any $\gamma''\in\Gamma''\setminus\mathrm{dom}(\phi\circ\phi')$, we also see that
\[((\beta'\bullet\beta)_{\rho_{\phi\circ\phi'}}^{\rho'}\!\circ(\phi\circ\phi')^\rho)(a)(\gamma'')=0_{\gamma''}=(\beta'^{\rho''}_{\rho'_{\phi'}}\!\circ\phi'^{\rho'}\circ\beta_{\rho_\phi}^{\rho'}\!\circ\phi^\rho(a))(\gamma'').\]
This shows that $(\beta'\bullet\beta)_{\rho_{\phi\circ\phi'}}^{\rho'}\!\circ(\phi\circ\phi')^\rho=\beta'^{\rho''}_{\rho'_{\phi'}}\!\circ\phi'^{\rho'}\circ\beta_{\rho_\phi}^{\rho'}\!\circ\phi^\rho$ and hence
\begin{align*}
\mathsf{A}((\rho'',\beta',\phi',\rho')(\rho',\beta,\phi,\rho))&=\mathsf{A}(\rho'',\beta'\bullet\beta,\phi\circ\phi',\rho)\\
&=(\mathcal{C}_\mathsf{r}^{\rho''}),(\beta'\bullet\beta)_{\rho_{\phi\circ\phi'}}^{\rho'}\!\circ(\phi\circ\phi')^\rho,\mathcal{C}_\mathsf{r}^\rho)\\
&=(\mathcal{C}_\mathsf{r}^{\rho''},\beta'^{\rho''}_{\rho'_{\phi'}}\!\circ\phi'^{\rho'}\circ\beta_{\rho_\phi}^{\rho'}\!\circ\phi^\rho,\mathcal{C}_\mathsf{r}^\rho)\\
&=(\mathcal{C}_\mathsf{r}^{\rho''},\beta'^{\rho''}_{\rho'_{\phi'}}\!\circ\phi'^{\rho'},\mathcal{C}_\mathsf{r}^{\rho'})(\mathcal{C}_\mathsf{r}^{\rho'},\beta_{\rho_\phi}^{\rho'}\!\circ\phi^\rho,\mathcal{C}_\mathsf{r}^\rho)\\
&=\mathsf{A}(\rho'',\beta',\phi',\rho')\mathsf{A}(\rho',\beta,\phi,\rho).
\end{align*}
Also $\mathsf{A}$ takes units in $\mathbf{Fell}$ to units in $\mathbf{C^*}$ and hence $\mathsf{A}$ is a functor.
\end{proof}

We might call $\mathsf{A}$ the \emph{abstraction} or \emph{algebraisation} functor.  If we restrict to certain Fell bundles (e.g. over locally compact Hausdorff \'etale groupoids), $\mathsf{A}$ will be faithful but certainly not full.  Indeed, C*-algebra homomorphisms resulting from Fell morphisms automatically preserve other structures arising from the Fell bundle, structures that are not determined by the C*-algebra alone, like the semigroups of slice-supported sections.  This suggests that, if we want an abstract counterpart to $\mathbf{Fell}$, we should take these additional structures as an intrinsic part of the C*-algebras forming the abstract category.  This leads to the notion of a `structured C*-algebra', as discussed in \cite{Bice2021DHKR} and as we intend to elaborate on in future work.

\bibliography{maths}{}
\bibliographystyle{alphaurl}

\end{document}